\documentclass[11pt,a4paper]{amsart}
\usepackage[DIV14,paper=a4,BCOR14mm,headinclude]{typearea}
\usepackage[british]{babel}
\usepackage{amssymb}
\usepackage{graphicx,version}
\usepackage{epstopdf}
\usepackage{color}
\usepackage{tikz}
\usetikzlibrary{patterns}
\usepackage{diagbox}
\usepackage{pgfplots}
\usepackage{pgfplotstable}
\usepackage{enumitem}
\usepackage{hyperref}
\usepackage{fp-eval}
\usepackage{esint}
\usepackage{subfigure}
\usepackage{caption}
\usepackage{mathabx}
\usepackage{booktabs}
\usepackage{todonotes}
\usepackage{subdepth}
\usepackage{amsmath, bm}
\definecolor{darkblue}{rgb}{0,0,.7}
\hypersetup{colorlinks=true,linkcolor=darkblue,citecolor=darkblue}

\newlist{alphenum}{enumerate}{1}
\setlist[alphenum]{fullwidth,label={(\alph*)}}

\allowdisplaybreaks[0]
\theoremstyle{definition}
\newtheorem{theorem}{Theorem}[section]

\newtheorem{lemma}[theorem]{Lemma}
\theoremstyle{remark}
\newtheorem{rem}{\bf Remark}

\numberwithin{figure}{section}
\numberwithin{table}{section}
\numberwithin{equation}{section}

\begin{document}

\title[Higher-order consistent splitting schemes for NSEs]{Stability and error analysis of a new class of  higher-order consistent splitting schemes for the Navier-Stokes equations}

\author[
	F. Huang and J. Shen
	]{
	Fukeng Huang$^\dag$ and Jie Shen$^{\dag,\ddag}$
		}
	\thanks{\noindent $^\dag$School of Mathematical Sciences, Eastern Institute of Technology, Ningbo, Zhejiang 315200, P. R. China (fkhuang@eitech.edu.cn, jshen@eitech.edu.cn). \\
$^\ddag$ Corresponding author. The work of J.S. is partially supported by NSFC grants W2431008 and 12371409.
	}

\begin{abstract}
  A new class of fully decoupled consistent splitting schemes for the Navier-Stokes equations are constructed and analyzed in this paper. The schemes are based on the Taylor expansion at $t^{n+\beta}$ with $\beta\ge 1$ being a free parameter. It is shown that by choosing {\color{black} $\beta= 3, \,6,\,9$} respectively for the second-, third- and fourth-order schemes, their numerical solutions are uniformed bounded in a strong norm, and admit optimal global-in-time convergence rates in both 2D and 3D. {\color{black}These } results are the first stability and convergence results for any fully decoupled, higher than second-order schemes for the Navier-Stokes equations. Numerical results are provided to show that the third- and fourth-order schemes based on the usual BDF (i.e. $\beta=1$) are not unconditionally stable while the new third- and fourth-order schemes with suitable $\beta$ are unconditionally stable and lead to expected convergence rates.
\end{abstract}


\maketitle

\section{Introduction}
We consider in this paper the construction and error analysis of a new class of high order consistent splitting schemes for the incompressible Navier-Stokes equations:
\begin{subequations}\label{NS}
\begin{align}
& \frac{\partial \bm u}{\partial  t}+\bm u \cdot {\nabla}  \bm u-\nu \Delta \bm u+\nabla p=\bm f,\label{NS1}\\
& \nabla \cdot \bm u=0,
\end{align}
\end{subequations}
with suitable initial conditions in a bounded domain  $\Omega\subset \mathbb{R}^d\;(d=2,3)$ and no-slip boundary condition $\bm u=0$ on $\partial \Omega$, and $\bm f$ is an external force.

The Navier-Stokes equations play an important role in many fields of science and engineering.
Due to its  importance  in applications, there is an enormous amount of work devoted to the numerical approximation of the Navier-Stokes equations.  These numerical methods  can be roughly classified into two categories: coupled approach with a mixed formulation (cf. \cite{brezzi2012mixed,elman2014finite,girault1979finite} and the references therein),  and decoupled approach through a  projection type method (including the pressure-correction  and the velocity correction methods) \cite{weinan1995projection,guermond2012convergence, guermond2011error,guermond2003velocity,guermond2004error, karniadakis1991high,orszag1986boundary,prohl1997projection,shen1992error,Shen2012Modeling,timmermans1996approximate,wang2000convergence}, and the consistent splitting method \cite{Shen03JCP,johnston2004accurate,shen2007error,WHSJCP22} (see also the gauge method \cite{E2003Gauge,Nochetto2003Gauge}).  We refer to \cite{guermond2006overview} for a review on the decoupled approach, and would like to point out that the  projection type schemes suffer from a splitting error which prevents them  from achieving full order accuracy in strong norms,
while the consistent splitting schemes do not loose accuracy. However, it has been a long standing open question on how to construct unconditionally stable second- or higher-order decoupled scheme with a rigorous stability and error analysis.

In a recent work \cite{HS2023}, we constructed a new second-order consistent splitting scheme, based on the Taylor expansions at time $t^{n+\beta}$, which, in the absence of nonlinear term,  reads as follows:
\begin{eqnarray}\label{BDF2NS1}
&\frac{(2\beta+1)\bm u^{n+1}-4\beta \bm u^n+(2\beta-1)\bm u^{n-1}}{2\delta t}-\nu{\color{black}\Delta}(\beta\bm u^{n+1}-(\beta-1)\bm u^n)+\nabla((\beta+1)p^n-\beta p^{n-1})={\color{black} f^{n+\beta} },\\
& (\nabla p^{n+1},\nabla q)={\color{black}(f^{n+1}, \nabla q)}-\nu(\nabla\times\nabla\times \bm u^{n+1},\nabla q),\quad\forall q\in H^1(\Omega),
\label{BDF2NS2}
\end{eqnarray}
where we use the identity $\Delta \bm u=\nabla \nabla \cdot \bm u-\nabla \times \nabla \times \bm u$ in \eqref{BDF2NS2}. Note that by integration by parts,  we can  express the volume integral in  the last equation as a boundary integral
\begin{equation*}
(\nabla\times\nabla\times \bm u^{n+1},\nabla q)= \int_{\partial\Omega} \bm n \times\nabla\times \bm u^{n+1}\cdot \nabla q,
\end{equation*}
{\color{black}  which makes it possible to  implement with $C^0$ finite-element methods.}
We were able to prove  that the above scheme with $\beta=5$ is unconditionally  stable in a strong norm, which was the  first such result  for any fully decoupled second- or higher-order scheme for the time dependent Stokes problem.  Then, by employing the generalized scalar auxiliary variable (GSAV) approach \cite{HS22} to handle the nonlinear term, we also conducted a rigorous stability  and error analysis for a corresponding second-order consistent splitting scheme for the Navier-Stokes equations.

While one can construct formally higher-order consistent splitting schemes based on the Taylor expansions at time $t^{n+\beta}$, it is an open question on how to prove its unconditional stability with a suitable $\beta$ for third- and higher-order schemes. 
 The main purpose of this paper is to provide an affirmative answer to this open question. More precisely,  our main contributions include:
 \begin{itemize}
 \item We  improve the results in \cite{HS2023} by showing that the second order consistent scheme based on the Taylor expansion at time $t^{n+3}$, instead of $t^{n+5}$, is   unconditionally stable in $l^2(H^2)\cap l^\infty(H^1)$. Note that as $\beta$ increases, so does the truncation error. Therefore, it is beneficial to use smaller $\beta$ when possible.
 \item We show that the third-order (resp.  fourth-order) consistent splitting schemes based on the Taylor expansion at time $t^{n+6}$ (resp. $t^{n+9}$) is unconditionally stable in $l^2(H^2)\cap l^\infty(H^1)$, and also carry out a rigorous error analysis with   global-in-time optimal error estimates both in 2D and 3D for the new second- to fourth-order consistent splitting schemes. Note that in \cite{HS2023} only local-in-time error estimate was established in 3D for a second-order consistent splitting scheme with $\beta=5$.
 \end{itemize}
 To the best of our knowledge, these schemes are the first higher than second-order  fully decoupled schemes for the  Navier-Stokes equations with a rigorous stability and error analysis. 

 We emphasize that the analysis in  \cite{HS2023} for the second-order scheme cannot be easily extended to third- or higher-order schemes. A main difficulty is that stability in the higher-order cases cannot be derived with usual  test functions.
 We recall that the stability of the usual higher-order BDF schemes for parabolic type equations relies on a result by Nevanlinna and Odeh \cite{nevanlinna1981}   (see also \cite{akrivis2021energy} for the extension to the six-order BDF scheme) in which the existence of suitable multipliers that can lead to energy stability was established. Most recently in \cite{HS2024}, we  extended the Nevanlinna and Odeh Lemma to the generalized  higher-order  (up to order four) BDF schemes for  parabolic type equations and carried out a rigorous error analysis. The technique used  in  \cite{HS2024} to identify suitable multipliers, as well as the Lemma on the Stokes commutator in \cite{liuCPAM}, are the two essential tools in proving the unconditional stability of the new schemes proposed in this paper. However,
unlike the parabolic type equations considered in   \cite{HS2024},  there is another  essential difficulty to control the explicit treatment of the pressure in the consistent splitting schemes. In fact,
 the multipliers identified in  \cite{HS2024} for parabolic type equations cannot be directly used here. A key and nontrivial step is to split the viscous term into suitable forms (see \eqref{splitB}) such that the explicit pressure terms can be controlled. 

The rest of the paper is organized as follows. In the next section,  we provide some preliminaries to be used in the sequel. In Section 3, we construct a new consistent splitting scheme for the time dependent Stokes equations and prove its unconditional stability in a strong norm. Then, in Section 4, we present the new high order consistent splitting scheme for the Navier-Stokes equations with explicit treatment for the non-linear terms and present detailed error analysis. In the final section, we provide a numerical example to validate the accuracy of our scheme, and conclude with a few  remarks.
\section{Preliminaries}
We first  introduce some notations. Let $W$ be a Banach space, we shall also use the standard notations
  $L^p(0,T;W)$ and $C([0,T];W)$. To simplify the notation, we often omit the spatial dependence  for the exact solution $u$, i.e.,  $u(x,t)$ is often denoted by $u(t)$.  We shall use bold faced letters to denote vectors and vector spaces, and use $C$ to denote a generic positive constant independent of the discretization parameters.  We denote by $(\cdot, \cdot)$ and $\|\cdot\|_0$ the inner product and the norm in $L^2(\Omega)$, {\color{black} and $\|\cdot\|_1$, $\|\cdot\|_2$, the norm in $H^1(\Omega)$, $H^2(\Omega)$ respectively, and denote
  \begin{equation*}
  \mathbb{V} = \left\{ \mathbf{v} \in \mathbf{H}^1_0(\Omega) \;:\; \nabla \cdot \mathbf{v} = 0 \right\}.
  \end{equation*}}

Next, we define the trilinear form $b(\cdot,\cdot,\cdot)$ by
\begin{equation*}
b(\bm u, \bm v, \bm w)= \int_{\Omega}(\bm u \cdot \nabla) \bm v \cdot \bm w d\bm x,
\end{equation*}
Using  H\"older inequality and Sobolev inequality, we have \cite{Tema83}
\begin{equation}\label{eq:ineq2d}
b(\bm u,\bm v,\bm w) \le c\|\bm u\|_1\| \bm v\|_1^{1/2}\|\bm v\|_2^{1/2}\|\bm w\|,  \quad  d=2,3.
 \end{equation}
We also use frequently the following inequalities ({\color{black} see, for instance, Lemma 2.1 in} \cite{Tema83}):
\begin{equation}\label{eq:ineq2}
b(\bm u,\bm v,\bm w)\le
\left\{
\begin{array}{lr}
  c\|\bm u\|_1\|\bm v\|_1\|\bm w\|_1;\\
  c\|\bm u\|_2\|\bm v\|_0\|\bm w\|_1;\\
 c\|\bm u\|_2\|\bm v\|_1\|\bm w\|_0;\\
 c\|\bm u\|_1\|\bm v\|_2\|\bm w\|_0;\\
c\|\bm u\|_0\|\bm v\|_2\|\bm w\|_1;
\end{array}
\right. \quad d\le {\color{black}3}.
\end{equation}
{\color{black}Note that the above inequalities, except the third one, are also valid when $d=4$.}

We will frequently use the following discrete versions of the Gronwall lemma.
\begin{lemma}\label{Gron2}
\textbf{(Discrete Gronwall Lemma) }{\color{black}(see, for instance, lemma 5.4 in \cite{he1998NM})} Let $a_n,\,b_n,\,c_n,$ and $d_n$ be four nonnegative sequences satisfying \begin{equation*}
a_m+\tau \sum_{n=1}^{m} b_n \le \tau \sum_{n=0}^{m-1}a_n d_n +\tau \sum_{n=0}^{m-1} c_n+ C, \,m \ge 1,
\end{equation*}
where $C$ and $\tau$ are two positive constants.
Then
\begin{equation*}
a_m+\tau \sum_{n=1}^{m} b_n \le \exp\big(\tau \sum_{n=0}^{m-1} d_n \big)\big(\tau \sum_{n=0}^{m-1}c_n+{\color{black} \tilde{C}} \big),\,m \ge 1,
\end{equation*}{\color{black} where $\tilde{C}$ is a constant that depends on the initial data $a_0$, $b_0$, $c_0$, and the constant $C$.}
\end{lemma}

In order to establish an unconditional stability result for \eqref{BDF2NS1}-\eqref{BDF2NS2}, we need the following result about the Stokes pressure introduced in \cite{liuCPAM}. For any $\bm u \in H^2(\Omega, \mathbb{R}^N)$, the Stokes pressure $p_s=p_s(\bm u)$ is defined as
\begin{equation}\label{stokespredef}
\nabla p_s(\bm u)=(\Delta \mathcal{P}-\mathcal{P} \Delta) \bm u,
\end{equation}
 where $\mathcal{P}$ is the Leray-Helmholtz projection operator onto divergence-free fields with zero normal component, providing the Helmholtz decomposition $\bm u =\mathcal{P} \bm u+\nabla \phi$, where
\begin{equation}
\big( \mathcal{P}\bm u, \nabla q \big)=\big(\bm u-\nabla \phi, \nabla q \big)=0, \quad \forall q \in H^1(\Omega).
\end{equation}
Then it is proved in \cite{liuCPAM} that
\begin{lemma}\label{stokespre}
Let $\Omega \subset \mathbb{R}^N (N\ge 2)$ be a connected bounded domain with $C^3$ boundary. Then for any $\varepsilon>0$, there exists $C \ge 0$ such that for all vector fields $\bm u \in H^2 \cap H_0^1(\Omega, \mathbb{R}^N)$,
\begin{equation}\label{stokespre2}
\int_{\Omega}|(\Delta \mathcal{P}-\mathcal{P}\Delta) \bm u|^2 \le \big(\frac{1}{2}+\varepsilon \big)\int_{\Omega}|\Delta \bm u|^2+C\int_{\Omega}|\nabla \bm u|^2.
\end{equation}
\end{lemma}

In order to make use of the energy techniques to conduct stability and error analysis, we need to find suitable multipliers with the help of following lemma from Dahlquist's G-stability theory \cite{dahlquist1978g}.
\begin{lemma}\label{Gstability}
Let $\alpha(\zeta)=\alpha_q\zeta^q+...+\alpha_0$ and $\mu(\zeta)=\mu_q\zeta^q+...+\mu_0$ be polynomials of degree at most $q$ (and at least one of them of degree $q$) that have no common divisors. Let $(\cdot, \cdot)$ be an inner product with associated norm $|\cdot|$. If
\begin{equation}
\text{Re}\frac{\alpha(\zeta)}{\mu(\zeta)}>0 \quad \text{for} \, |\zeta|>1,
\end{equation}
then there exists a symmetric positive definite matrix $G=(g_{ij}) \in \mathbb{R}^{q\times q}$ and real $\delta_0,...,\delta_q$ such that for $\upsilon^0,...,\upsilon^q$ in the inner product space,
\begin{equation}
\big(\sum_{i=0}^q \alpha_i \upsilon^i, \sum_{j=0}^q \mu_j \upsilon^j \big)=\sum_{i,j=1}^{q}g_{ij}(\upsilon^i,\upsilon^j)-\sum_{i,j=1}^{q}g_{ij}(\upsilon^{i-1},\upsilon^{j-1})+\big |\sum_{i=0}^q\delta_i \upsilon^i \big|^2.
\end{equation}

\end{lemma}

\section{Higher-order  consistent splitting scheme for the time dependent Stokes equations}
We shall first present   generalized  BDF  consistent splitting schemes based on the Taylor expansion at time $t^{n+\beta}$, and then  show that the $k$-th ($k=2,3,4$) order with suitable $\beta$s are unconditionally stable in the strong norm.
\subsection{The generalized BDF  schemes}
We note that we constructed  in \cite{HS2024} generalized  BDF schemes  based on the Taylor expansion at time $t^{n+\beta}$ for general parabolic type equations.   Following  \cite{HS2024}, we can construct  generalized  BDF  consistent splitting schemes as follows.
Given an integer $k\ge 2$, denote $t^n= n\delta t$, it follows from the Taylor expansion at time $t^{n+\beta}$ that
\begin{equation}\label{Taylor}
\phi(t^{n+1-i})=\sum_{m=0}^{k-1}[(1-i-\beta)\delta t]^m \frac{\phi^{(m)}(t^{n+\beta})}{m!}+\mathcal{O}(\delta t^{k}),\quad \forall i \ge 0.
\end{equation}
Therefore we have
\begin{equation}\label{TaylorA}
\frac{1}{\delta t}\sum_{q=0}^{k}a_{k,q}(\beta)\phi(t^{n+1-k+q})=\partial_t\phi(t^{n+\beta})+\mathcal{O}(\delta t^k),
\end{equation}
with $a_{k,q}(\beta)$ can be obtained by solving the linear system:
\begin{equation}\label{solve_akq}
{\left[
\begin{array}{ccccc}
1&1&...&...&1\\
\beta-1&\beta&...&...&\beta+k-1\\
(\beta-1)^2&\beta^2&...&...&(\beta+k-1)^2\\
\vdots&\vdots&\vdots&\vdots&\vdots\\
(\beta-1)^k&\beta^k&...&...&(\beta+k-1)^k
\end{array}
\right]}{\left[
\begin{array}{c}
a_{k,k}(\beta)\\a_{k,k-1}(\beta)\\a_{k,k-2}(\beta)\\ \vdots \\a_{k,0}(\beta)
\end{array}
\right]}=
{\left[
\begin{array}{c}
0\\-1\\0\\ \vdots \\0
\end{array}
\right]};
\end{equation}
and
\begin{equation}\label{TaylorB}
\sum_{q=0}^{k-1}b_{k,q}(\beta)\phi(t^{n+2-k+q})=\phi(t^{n+\beta})+\mathcal{O}(\delta t^k),
\end{equation}
with $b_{k,q}(\beta)$ can be obtained by solving the linear system:
\begin{equation}\label{solve_bkq}
{\left[
\begin{array}{ccccc}
1&1&...&...&1\\
\beta-1&\beta&...&...&\beta+k-2\\
\vdots &\vdots&\vdots&\vdots&\vdots\\
(\beta-1)^{k-1}&\beta^{k-1}&...&...&(\beta+k-2)^{k-1}
\end{array}
\right]}{\left[
\begin{array}{c}
b_{k,k-1}(\beta)\\b_{k,k-2}(\beta)\\\vdots\\b_{k,0}(\beta)
\end{array}
\right]}=
{\left[
\begin{array}{c}
1\\0\\\vdots\\0
\end{array}
\right]};
\end{equation}
and finally
\begin{equation}\label{TaylorC}
\sum_{q=0}^{k-1}c_{k,q}(\beta)\phi(t^{n+1-k+q})=\phi(t^{n+\beta})+\mathcal{O}(\delta t^k),
\end{equation}
with $c_{k,q}(\beta)$ can be obtained by solving the linear system:
\begin{equation}\label{solve_ckq}
{\left[
\begin{array}{ccccc}
1&1&...&...&1\\
\beta&\beta+1&...&...&\beta+k-1\\
\vdots &\vdots&\vdots&\vdots&\vdots\\
\beta^{k-1}&(\beta+1)^{k-1}&...&...&(\beta+k-1)^{k-1}
\end{array}
\right]}{\left[
\begin{array}{c}
c_{k,k-1}(\beta)\\c_{k,k-2}(\beta)\\\vdots\\c_{k,0}(\beta)
\end{array}
\right]}=
{\left[
\begin{array}{c}
1\\0\\\vdots\\0
\end{array}
\right]}.
\end{equation}
Next, we would like to introduce the following notations to simplify the presentation below,
\begin{equation}\label{ABC_def}
\quad A_k^{\beta}(\phi^{i})=\sum_{q=0}^{k}a_{k,q}(\beta)\phi^{i-k+q},\,\quad B_k^{\beta}(\phi^{i})=\sum_{q=0}^{k-1}b_{k,q}(\beta)\phi^{i-k+1+q},\,\quad C_k^{\beta}(\phi^{i})=\sum_{q=0}^{k-1}c_{k,q}(\beta) \phi^{i-k+1+q}.
\end{equation}
Now, with the above notations,  the generalized $k$-th order BDF type schemes with explicit treatment of the pressure for the time dependent Stokes equation (in the absence of $\bm f$ and nonlinear term in \eqref{NS})  are as follows:
\begin{subequations}\label{scheme_stoke}
\begin{align}
&\frac{A_k^{\beta}(\bm u^{n+1})}{\delta t}-\nu \Delta B_k^{\beta}(\bm u^{n+1})+\nabla C_k^{\beta}(p^{n})=0, \label{scheme_stoke1}\\
&(\nabla p^{n+1},\nabla q)=-\nu(\nabla\times\nabla\times \bm u^{n+1},\nabla q),\quad\forall q\in H^1(\Omega).\label{scheme_stoke2}
\end{align}
\end{subequations}
\subsection{Linear stability regions}
Before providing the stability proof for the new schemes \eqref{scheme_stoke}, we would like to first investigate the linear stability regions of the new BDF type schemes. For the test equation $\phi_t= \lambda\phi$, by performing the Taylor expansions at $t^{n+\beta}$, a more  general BDF type method can be written as
\begin{equation}\label{genBDF}
\frac{A_k^{\beta}(\phi^{n+1})}{\delta t}=\lambda B_k^{\beta}(\phi^{n+1}).
\end{equation}
In order to study the stability region for  $\beta \ne 1$, we set $\phi^n=\mu^n$ and ${z=\lambda\delta t}$ in \eqref{genBDF} to obtain its characteristic polynomial
\begin{equation}\label{eq:chara}
\sum_{q=0}^k (a_{k,q}(\beta)- b_{k,q-1}(\beta)z)\mu^q=0,
\end{equation}
where $a_{k,q}(\beta)$ and $b_{k,q}(\beta)$ are defined in \eqref{solve_akq} and \eqref{solve_bkq} respectively and we further define $b_{k,-1}=0$ in \eqref{eq:chara}.
{\color{black} Then the region of absolute stability is the set of all $ z \in \mathbb{C}$ such that all roots $\mu$ of the characteristic equation \eqref{eq:chara} satisfy  $|\mu| \leq 1$, and any root with $|\mu| = 1$ is simple.}
In Table. \ref{stability}, we plot  the stability regions of the general BDF type method  \eqref{genBDF} for $\beta=1,3,6,9$.
 We observe that  the  stability regions increases as we increases $\beta$, at the expense of increased truncation error.
\begin{table}[h]
\centering
\begin{tabular}{c|cccc}
\hline
& $\beta=1$ &$\beta=3$ & $\beta=6$ & $\beta=9$\\
\hline
Second order& \begin{minipage}[b]{0.18\columnwidth}
		\centering
		\raisebox{-.5\height}{\includegraphics[width=0.95\linewidth]{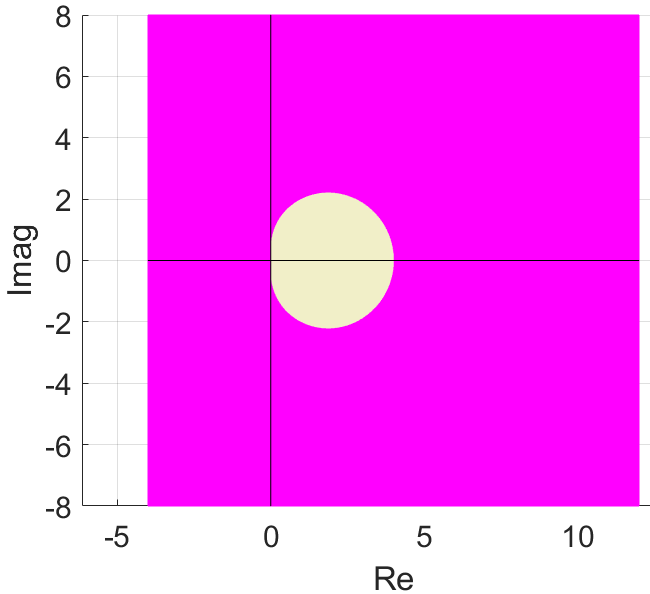}}
	\end{minipage}
&\begin{minipage}[b]{0.18\columnwidth}
		\centering
		\raisebox{-.5\height}{\includegraphics[width=0.95\linewidth]{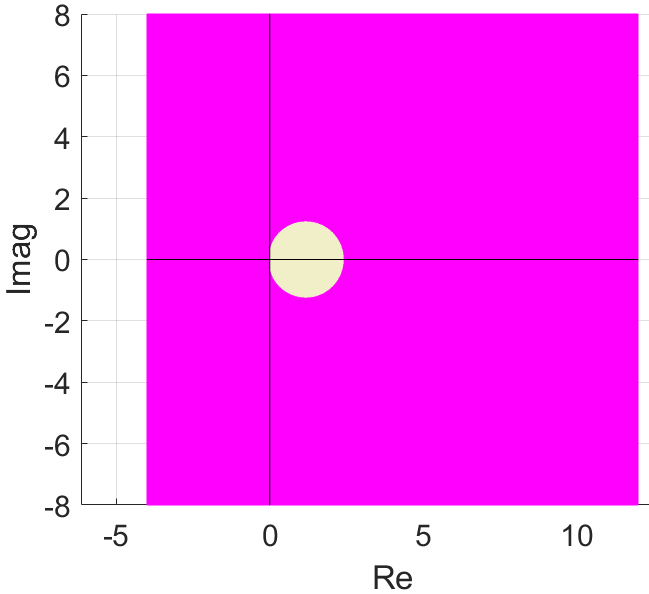}}
	\end{minipage}
&\begin{minipage}[b]{0.18\columnwidth}
		\centering
		\raisebox{-.5\height}{\includegraphics[width=0.95\linewidth]{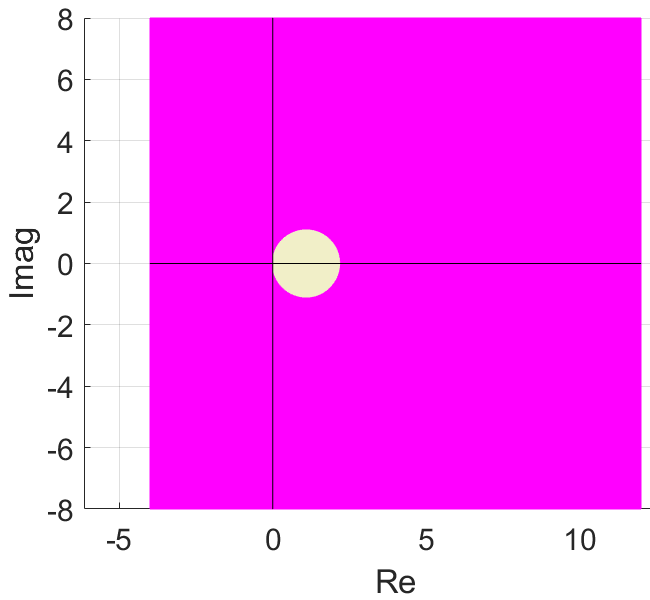}}
	\end{minipage}
&\begin{minipage}[b]{0.18\columnwidth}
        \centering
		\raisebox{-.5\height}{\includegraphics[width=0.95\linewidth]{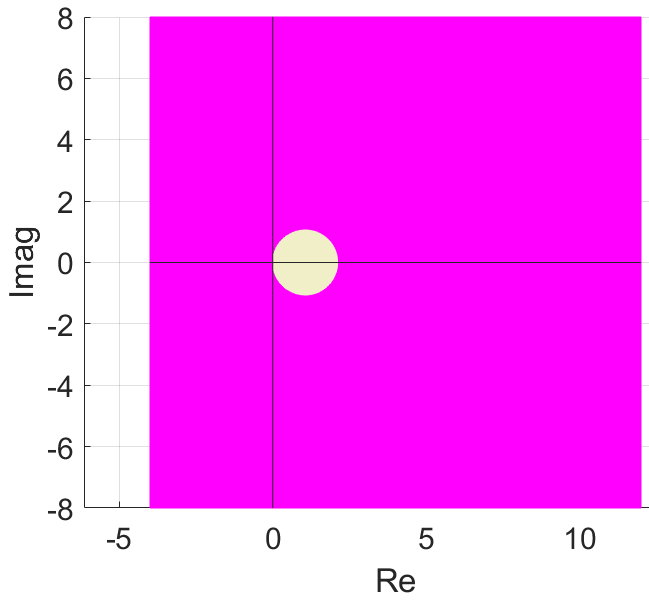}}
	\end{minipage}\\
\hline
Third order & \begin{minipage}[b]{0.2\columnwidth}
		\centering
		\raisebox{-.5\height}{\includegraphics[width=0.95\linewidth]{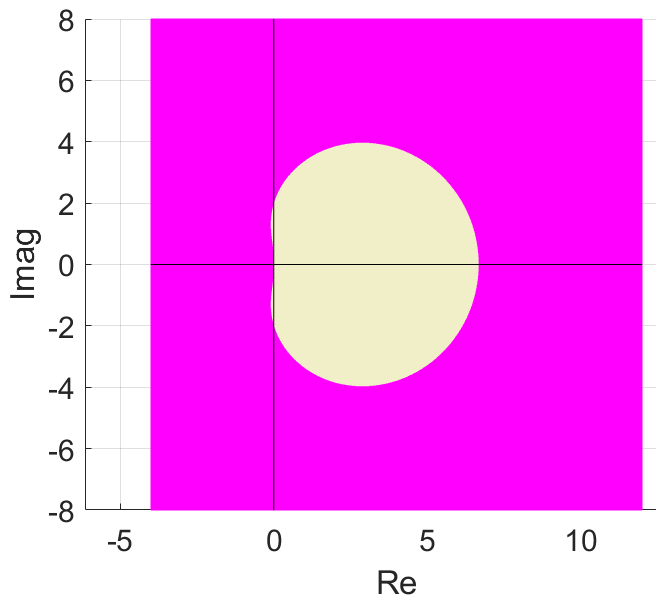}}
	\end{minipage}
&\begin{minipage}[b]{0.18\columnwidth}
		\centering
		\raisebox{-.5\height}{\includegraphics[width=0.95\linewidth]{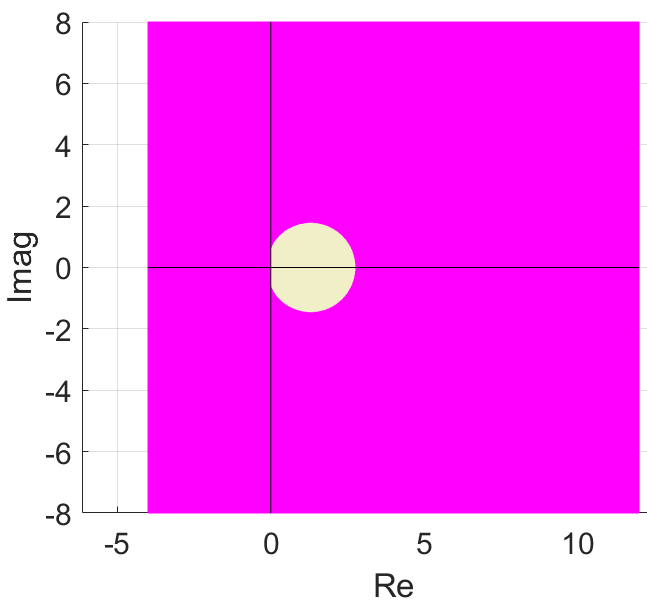}}
	\end{minipage}
&\begin{minipage}[b]{0.18\columnwidth}
		\centering
		\raisebox{-.5\height}{\includegraphics[width=0.95\linewidth]{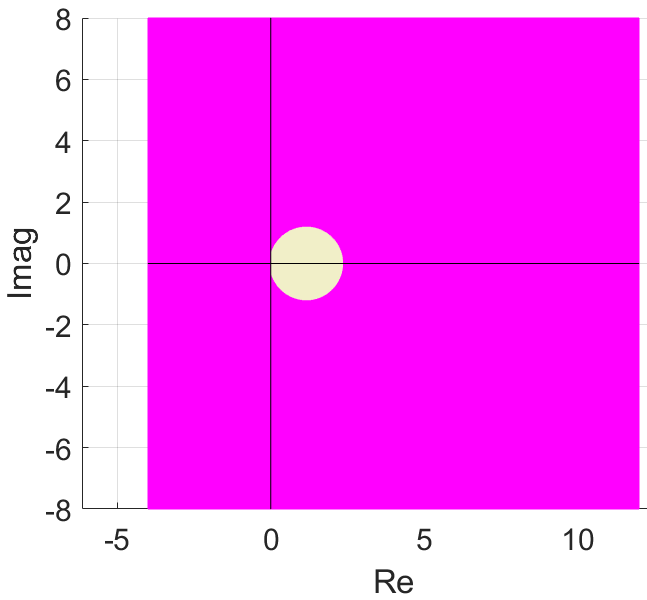}}
	\end{minipage}
&\begin{minipage}[b]{0.18\columnwidth}
		\centering
		\raisebox{-.5\height}{\includegraphics[width=0.95\linewidth]{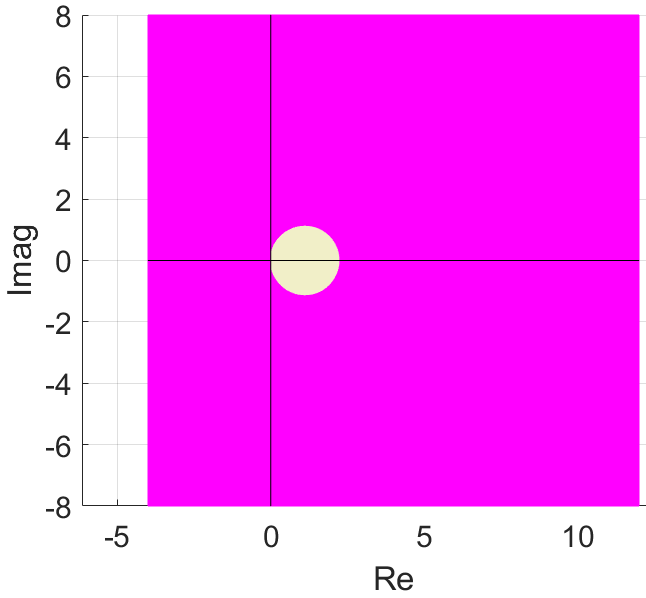}}
	\end{minipage}\\
\hline
Fourth order & \begin{minipage}[b]{0.18\columnwidth}
		\centering
		\raisebox{-.5\height}{\includegraphics[width=0.95\linewidth]{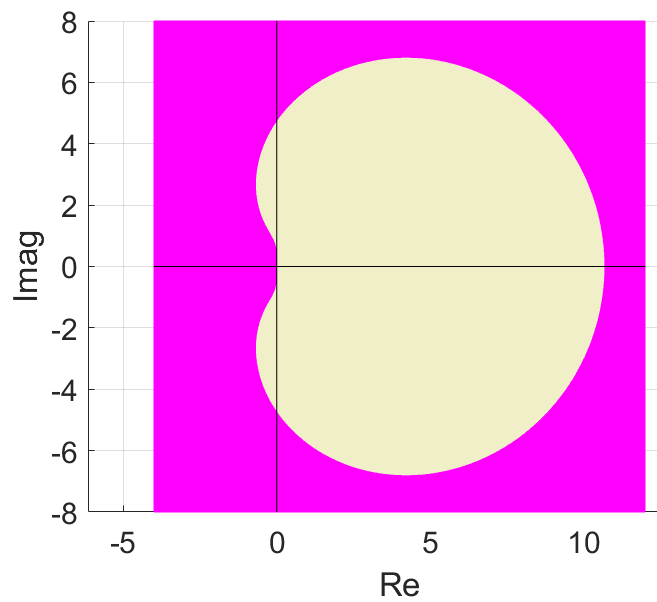}}
	\end{minipage}
&\begin{minipage}[b]{0.18\columnwidth}
		\centering
		\raisebox{-.5\height}{\includegraphics[width=0.95\linewidth]{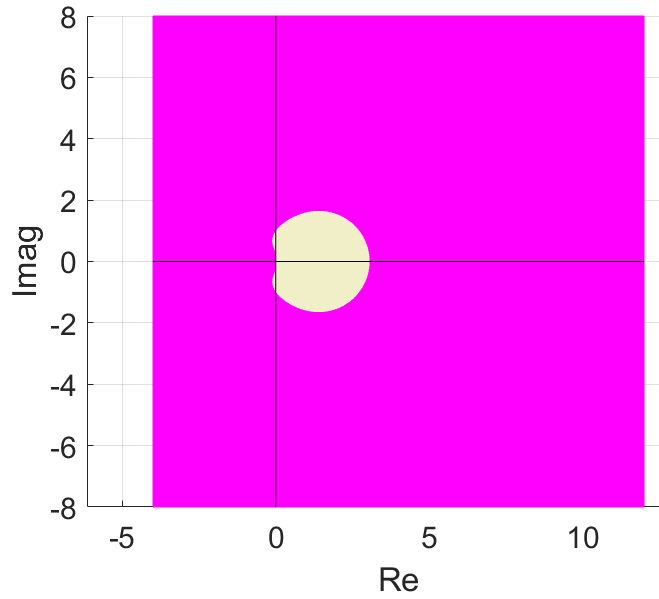}}
	\end{minipage}
&\begin{minipage}[b]{0.18\columnwidth}
		\centering
		\raisebox{-.5\height}{\includegraphics[width=0.95\linewidth]{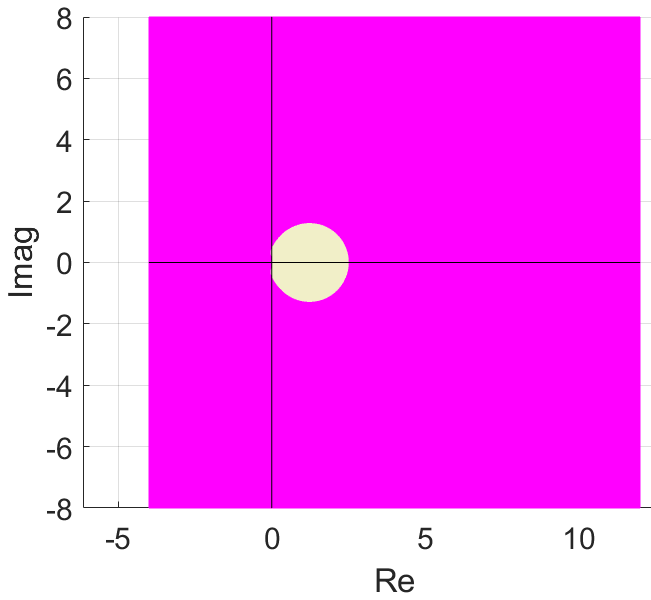}}
	\end{minipage}
&\begin{minipage}[b]{0.18\columnwidth}
		\centering
		\raisebox{-.5\height}{\includegraphics[width=0.95\linewidth]{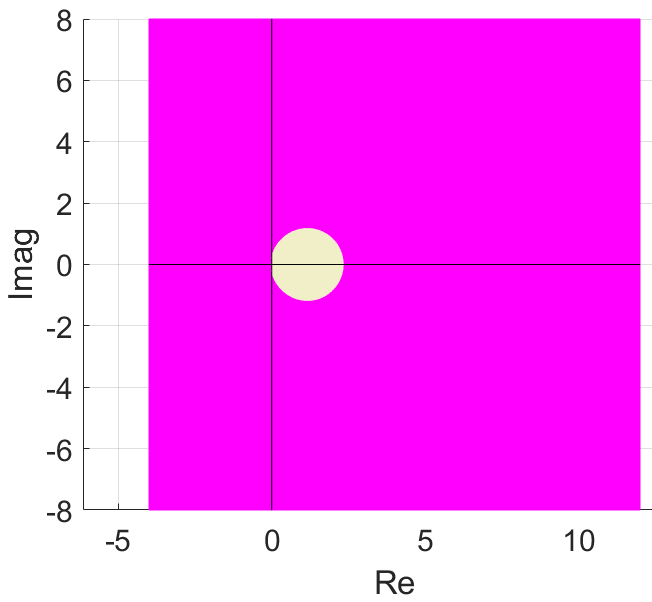}}
	\end{minipage}\\
\hline
\end{tabular}
\caption{The pink parts show the linear stability regions. }\label{stability}
\end{table}
\subsection{A uniform multiplier}
To conduct the stability and error analysis,  we need to choose a suitable $\beta$ for schemes of different orders. In the following, we choose $\beta=\beta_k$ as follows:
\begin{equation}\label{betak}
\beta_2=3,\,\beta_3=6,\,\beta_4=9.
\end{equation}
These choices of $\beta$ are sufficient for our purposes, but not necessarily the smallest possible. We recall that as $\beta$ increases, so does the truncation error. So it is desirable to choose $\beta$ as small as possible while maintaining stability.
 For the rest of the paper, we fix $\beta_k$ as \eqref{betak} and then the explicit expression of \eqref{scheme_stoke1} becomes:

\noindent $k=2$, $\beta=3$:
\begin{small}
\begin{equation}
\frac{7\bm u^{n+1}-12\bm u^n+5\bm u^{n-1}}{2\delta t}-\nu \Delta(3\bm u^{n+1}-2\bm u^n)+\nabla (4p^n-3p^{n-1})=0;
\end{equation}
\end{small}\noindent $k=3$, $\beta=6$:
\begin{small}
\begin{equation}
\frac{146\bm u^{n+1}-393\bm u^n+354\bm u^{n-1}-107\bm u^{n-2}}{6\delta t}-\nu \Delta(21\bm u^{n+1}-35\bm u^n+15\bm u^{n-1})+\nabla (28p^n-48p^{n-1}+21p^{n-2})=0;
\end{equation}
\end{small}\noindent $k=4$, $\beta=9$:
\begin{small}
\begin{equation}
\begin{split}
\frac{2289\bm u^{n+1}-8432\bm u^n+11700\bm u^{n-1}-7248\bm u^{n-2}+1691\bm u^{n-3}}{12\delta t}&-\nu \Delta(165\bm u^{n+1}-440\bm u^n+396\bm u^{n-1}-120\bm u^{n-2})\\
+&\nabla (220p^n-594p^{n-1}+540p^{n-2}-165p^{n-3})=0;
\end{split}
\end{equation}
\end{small}

A key step in the proof  is to properly split $B_k^{\beta_k}(\bm u^{n+1})$ into three parts as follows:
\begin{equation}\label{splitB}
B_k^{\beta_k}(\bm u^{n+1})=\eta_k C_k^{\beta_k}(\bm u^{n+1})+D_k^{\beta_k}(\bm u^{n+1})+F_k^{\beta_k}(\bm u^{n+1}),\quad k=2,3,4,
\end{equation}
with $\eta_k$ being a suitable positive number to be specified, and
\begin{subequations}\label{fkq}
\begin{align}
& F_2^{\beta_2}(\bm u^{n+1}):=\sum_{q=0}^{1}f_{2,q}(\beta_2) \bm u^{n+q}=\frac{1}{100}\bm u^{n+1}+0\bm u^{n},\\
& F_3^{\beta_3}(\bm u^{n+1}):=\sum_{q=0}^{2}f_{3,q}(\beta_3) \bm u^{n-1+q}=\frac{1}{100}(27\bm u^{n+1}-21\bm u^{n})+0\bm u^{n-1},\\
& F_4^{\beta_4}(\bm u^{n+1}):=\sum_{q=0}^{3}f_{4,q}(\beta_4) \bm u^{n-2+q}=\frac{2}{10^5}(215\bm u^{n+1}-375\bm u^{n}+165\bm u^{n-1})+0\bm u^{n-2},
\end{align}
\end{subequations}
and
\begin{equation}\label{dkq}
d_{k,q}(\beta_k)=b_{k,q}(\beta_k)-\eta_k c_{k,q}(\beta_k)-f_{k,q}(\beta_k),\quad D_k^{\beta_k}(\bm u^{n+1}):=\sum_{q=0}^{k-1}d_{k,q}(\beta_k) \bm u^{n-k+2+q}.
\end{equation}
The reasons for the above splitting will become clear later. In the above,
 $\eta_k$ should be chosen such that $\eta_k>\frac{\sqrt{2}}{2}\approx 0.7071$, the reason will be given in \eqref{etak_cond} below.

By choosing $F_k^{\beta_k}$ as in \eqref{fkq}, we have the following inequalities,  which are useful in the next section. The explicit telescoping terms given in appendix \ref{app1} imply there exits $U_k(\bm u^{i},...,\bm u^{i+2-k})\ge 0,\,k=2,3,4$ such that
\begin{equation}\label{Fk_ineq}
\big(F_k^{\beta_k}(\bm u^{n+1}), C_k^{\beta_k}(\bm u^{n+1}) \big) \ge \kappa_k \|\bm u^{n+1}\|^2+U_k(\bm u^{n+1},...,\bm u^{n+3-k})-U_k(\bm u^{n},...,\bm u^{n+2-k}),
\end{equation}
with
\begin{equation}\label{deltak}
\kappa_2=\frac{1}{100},\quad \kappa_3=\frac{3}{50},\quad \kappa_4=\frac{1}{10^4}.
\end{equation}

In the following, we fix $\eta_k=0.71$ and  $\beta_k$  as in \eqref{betak}  for $k=2,3,4$. Then, we can establish two important lemmas which play  key roles in the stability and error analysis. To this end, we introduce some polynomials with coefficients appearing in \eqref{ABC_def} and \eqref{dkq},
\begin{equation}\label{ACD_poly}
\tilde{A}_k^{\beta_k}(\zeta)=\sum_{q=0}^{k}a_{k,q}(\beta_k)\zeta^q,\,\quad \tilde{C}_k^{\beta_k}(\zeta)=\sum_{q=0}^{k-1}c_{k,q}(\beta_k)\zeta^q,\,\quad \tilde{D}_k^{\beta_k}(\zeta^q)=\sum_{q=0}^{k-1}d_{k,q}(\beta_k) \zeta^q.
\end{equation}
\begin{lemma}\label{AC_lemma}
Given $\tilde{A}^{\beta_k}_k(\zeta), \tilde{C}^{\beta_k}_k(\zeta)$ defined in \eqref{ACD_poly} and $\beta_k$ as in \eqref{betak}, we have
\begin{equation}\label{AC_lemma1}
\rm{gcd}\big(\tilde{A}^{\beta_k}_k(\zeta),\zeta\tilde{C}^{\beta_k}_k(\zeta)\big)=1,\,k=2,3,4,
\end{equation}
 i.e. they have no common divisor, and
\begin{equation}\label{AC_lemma2}
\text{Re}\frac{\tilde{A}_k^{\beta_k}(\zeta)}{\zeta\tilde{C}_k^{\beta_k}(\zeta)}\,>0, \quad \text{for}\, |\zeta|>1,\,k=2,3,4.
\end{equation}
\end{lemma}
The proof of the above lemma in a more general form was given in \cite{HS2024} (Theorem 1), which shows  \eqref{AC_lemma1} and \eqref{AC_lemma2} are true for all $\beta_k > 1$. 
\begin{lemma}\label{DC_lemma}
Given $\tilde{D}^{\beta_k}_k(\zeta), \tilde{C}^{\beta_k}_k(\zeta)$ defined in \eqref{ACD_poly}, $\beta_k$ as in \eqref{betak} and $\eta_k=0.71$, we have
\begin{equation}
\rm{gcd}\big(\tilde{D}^{\beta_k}_k(\zeta),\tilde{C}^{\beta_k}_k(\zeta)\big)=1,\, k=2,3,4,
\end{equation}
 i.e. they have no common divisor, and
\begin{equation}
\text{Re}\frac{\tilde{D}_k^{\beta_k}(\zeta)}{\tilde{C}_k^{\beta_k}(\zeta)}\,>0, \quad \text{for}\, |\zeta|>1,\,k=2,3,4.
\end{equation}
\end{lemma}
We shall defer   the proof to  Appendix \ref{app2}.

Several remarks are in order.
\begin{itemize}
\item One may choose other forms of  $F_k^{\beta_k}$ in \eqref{fkq}. As long as \eqref{Fk_ineq} with $\kappa_k>0$ and Lemma \ref{DC_lemma} are still true.
\item {\color{black} For larger values of $\eta_k$, a larger $\beta$ may be required to prove Lemma \ref{DC_lemma}, which in turn introduces a larger truncation error in the scheme. Therefore, we complete the proof by choosing $\eta_k = 0.71$-as small as possible while still satisfying $\eta_k > \frac{\sqrt{2}}{2}$.}
\item $\beta_k$ can also be non-integer, for example, one can prove the above two lemmas by the same processes by choosing $\beta_2=2.9$ for the second order scheme.
\end{itemize}

\subsection{Unconditional stability}
With the help of Lemma \ref{stokespre}--Lemma \ref{DC_lemma}, we can prove the following results for the scheme \eqref{scheme_stoke}.

\begin{theorem} {\color{black} Suppose $\Omega$ satisfies the conditions in Lemma \ref{stokespre}} and given  $\bm u^{i},\, i=1,..,k-1$ such that $\|\nabla \bm u^{i}\|^2+\delta t \|\Delta \bm u^{i}\|^2\le C \|\nabla \bm u^0\|^2,\, i=1,..,k-1$.
The scheme \eqref{scheme_stoke} with {\color{black} $\beta = \beta_k$ }chosen as in \eqref{betak} is unconditionally stable in the sense that, {\color{black} for all $n \ge 0$, we have}
\begin{equation}
\|\nabla \bm u^{n+1}\|^2+\delta t\sum_{i=k-1}^{n}\|\Delta C_k^{\beta_k}(\bm u^{i+1})\|^2+\delta t\sum_{i=0}^{n}\|\Delta \bm u^{i+1}\|^2 +\delta t \sum_{i=0}^{n} \|\nabla p^{i+1}\|^2 \le C,\,k=2,3,4,
\end{equation}
where $C$ is a constant independent of the time step $\delta t$ and $n$.
\end{theorem}

\begin{proof}
	
  Taking the inner product of \eqref{scheme_stoke1} with $-\Delta C_k^{\beta_k}(\bm u^{n+1})$, we deal with the three terms as follows. First, we split $B_k^{\beta_k}$ as in \eqref{splitB},
\begin{small}
\begin{equation}\label{eq1}
	\begin{split}
		& \big(-\nu\Delta B_k^{\beta_k}(\bm u^{n+1}),-\Delta C_k^{\beta_k}(\bm u^{n+1}) \big) =\nu \Big(\Delta (\eta_kC_k^{\beta_k}(\bm u^{n+1})+D_k^{\beta_k}(\bm u^{n+1})+F_k^{\beta_k}(\bm u^{n+1})), \Delta C_k^{\beta_k}(\bm u^{n+1}) \Big)\\
		&= \eta_k\nu \|\Delta C_k^{\beta_k}(\bm u^{n+1}) \|^2+\nu \big(\Delta D_k^{\beta_k}(\bm u^{n+1}),\Delta C_k^{\beta_k}(\bm u^{n+1})\big)+\nu \big(\Delta F_k^{\beta_k}(\bm u^{n+1}),\Delta C_k^{\beta_k}(\bm u^{n+1}) \big).
	\end{split}
\end{equation}
\end{small}If we choose $\eta_k=0.71$, it follows from Lemma \ref{Gstability} and Lemma \ref{DC_lemma} that there exists a symmetric positive definite matrix $H_k=(h_{ij})\in \mathbb{R}^{(k-1)\times (k-1)}$ such that
\begin{small}
\begin{equation}
\big(\Delta D_k^{\beta_k}(\bm u^{n+1}),\Delta C_k^{\beta_k}(\bm u^{n+1}) \big) \ge \sum_{i,j=1}^{k-1}h_{ij}(\Delta \bm u^{n+2+i-k},\Delta \bm u^{n+2+j-k})-\sum_{i,j=1}^{k-1}h_{ij}(\Delta \bm u^{n+1+i-k},\Delta \bm u^{n+1+j-k}),
\end{equation}
\end{small}
and \eqref{Fk_ineq} implies that
\begin{small}
\begin{equation}
\big(\Delta F_k^{\beta_k}(\bm u^{n+1}),\Delta C_k^{\beta_k}(\bm u^{n+1}) \big) \ge \kappa_k\|\Delta \bm u^{n+1}\|^2+U_k(\Delta \bm u^{n+1},...,\Delta \bm u^{n+3-k})-U_k(\Delta \bm u^{n},...,\Delta \bm u^{n+2-k}).
\end{equation}
\end{small}For the pressure term,
\begin{equation}\label{eq2}
		\big(\nabla C_k^{\beta_k}(p^n),\Delta C_k^{\beta_k}(\bm u^{n+1})  \big)
		 \le \frac{\gamma}{2\nu} \|\nabla C_k^{\beta_k}(p^n)\|^2+\frac{\nu}{2\gamma}\|\Delta C_k^{\beta_k}(\bm u^{n+1})\|^2,
\end{equation}
with $\gamma$ can be any positive number.
A key step is to deal with the first term in the above using Lemma \ref{stokespre}.
We recall from \cite{liuCPAM} that
\begin{equation}\label{stokesp}
(\nabla p_s(\bm u^n),\nabla q)=-(\nabla\times\nabla\times \bm u^n,\nabla q),
\end{equation}
where $p_s(\bm u^n)$ is the Stokes pressure associated with $\bm u^n$ and it follows from \eqref{scheme_stoke2} that
{\color{black}\begin{equation}\label{BDF2NS3}
\big(\nabla C_k^{\beta_k}(p^n),\nabla q\big)=-\nu\big(\nabla\times\nabla\times C_k^{\beta_k}(\bm u^n),\nabla q\big),\quad\forall q\in H^1(\Omega).
\end{equation}}
Taking $q=C_k^{\beta_k}(p^n)$ in \eqref{BDF2NS3} and in \eqref{stokesp}, we find from \eqref{stokesp} with $\bm u=C_k^{\beta_k}(\bm u^n)$ that
{\color{black}
\begin{equation}\label{stokesp2}
 \|\nabla C_k^{\beta_k}(p^n)\|\le\nu \|\nabla p_s(C_k^{\beta_k}(\bm u^n))\|.
\end{equation}}
Now, we can use \eqref{stokesp2} and \eqref{stokespre2} to bound the first term as follows
\begin{equation}\label{eq3}
\frac{\gamma}{2\nu}	\|\nabla  C_k^{\beta_k}(p^n)\|^2 \le \frac{\gamma\nu}{2} \|\nabla  p_s(C_k^{\beta_k}(\bm u^n))\|^2
	\le \gamma\nu (\frac{1}{4}+\frac{\varepsilon}{2})\| \Delta C_k^{\beta_k}(\bm u^n)\|^2 +C\gamma\nu \|\nabla C_k^{\beta_k}(\bm u^n)\|^2,
\end{equation}
with $\varepsilon >0$ which can be  arbitrarily small.
We observe from \eqref{eq1}-\eqref{eq3} that to ensure stability, we need
\begin{equation}\label{etak_cond}
	\eta_k \ge \mathop{\rm {min}}\limits_{\gamma>0}\big(\frac{1}{2\gamma}+\gamma(\frac{1}{4}+\frac{\varepsilon}{2})\big)\overset{\gamma=\sqrt{2}}{=}\frac{\sqrt{2}}{2}+\frac{\sqrt{2}\varepsilon}{2}.
\end{equation}
As $\varepsilon$ can be chosen arbitrarily small, we only need to choose $\eta_k>\frac{\sqrt{2}}{2}\approx 0.7071$ to ensure \eqref{etak_cond} and that is why we fix $\eta_k=0.71$.
It remains to deal with the last term:
\begin{equation}
	\frac1{\delta t}\big(A_k^{\beta_k}(\bm u^{n+1}),-\Delta C_k^{\beta_k}(\bm u^{n+1}) \big).
\end{equation}
Again, it follows from Lemma \ref{Gstability} and Lemma \ref{AC_lemma} that there exists symmetric positive definite matrix $G_k=(g_{ij})\in \mathbb{R}^{k\times k}$ such that
\begin{equation}\label{eq4}
\big(A_k^{\beta_k}(\bm u^{n+1}),-\Delta C_k^{\beta_k}(\bm u^{n+1}) \big) \ge \sum_{i,j=1}^{k}g_{ij}(\nabla \bm u^{n+1+i-k},\nabla \bm u^{n+1+j-k})-\sum_{i,j=1}^{k}g_{ij}(\nabla \bm u^{n+i-k},\nabla \bm u^{n+j-k}).
\end{equation}

With $\gamma=\sqrt{2}$ and $\eta_k=0.71$, summing up $\delta t(\eqref{eq1}+\eqref{eq2})$  and \eqref{eq4}, using the estimates above,  we find
\begin{small}
\begin{equation}\label{tosumA}
\begin{split}
& \sum_{i,j=1}^{k}g_{ij}(\nabla \bm u^{n+1+i-k},\nabla \bm u^{n+1+j-k})-\sum_{i,j=1}^{k}g_{ij}(\nabla \bm u^{n+i-k},\nabla \bm u^{n+j-k})+\nu \delta t\sum_{i,j=1}^{k-1}h_{ij}(\Delta \bm u^{n+2+i-k},\Delta \bm u^{n+2+j-k})\\
&-\nu\delta t\sum_{i,j=1}^{k-1}h_{ij}(\Delta \bm u^{n+1+i-k},\Delta \bm u^{n+1+j-k})+0.71\nu \delta t\|\Delta C_k^{\beta_k}(\bm u^{n+1}) \|^2+\nu \kappa_k\delta t\|\Delta \bm u^{n+1}\|^2\\
& +\delta t \nu U_k(\Delta \bm u^{n+1},...,\Delta \bm u^{n+3-k})-\delta t \nu U_k(\Delta \bm u^{n},...,\Delta \bm u^{n+2-k})\\
 \le&\frac{\sqrt{2}\nu\delta t}{4}\|\Delta C_k^{\beta_k}(\bm u^{n+1})\|^2+ (\frac{\sqrt{2}}{4}+\frac{\sqrt{2}\varepsilon}{2})\nu\delta t\| \Delta C_k^{\beta_k}(\bm u^n)\|^2 +\sqrt{2}C\nu \delta t\|\nabla C_k^{\beta_k}(\bm u^n)\|^2.
\end{split}
\end{equation}
\end{small}
Now, we can choose $\varepsilon$ small enough such that
\begin{equation}
0.71-\frac{\sqrt{2}}{2}-\frac{\sqrt{2}\varepsilon}{2}=\rho>0,
\end{equation}
 and take the sum of $n$ from $k-1$ to $m\le \frac T{\delta t}-1$ on \eqref{tosumA}. Dropping some unnecessary terms, we obtain
{\color{black} \begin{equation*}
\begin{split}
&  \sum_{i,j=1}^{k}g_{ij}(\nabla \bm u^{m+1+i-k},\nabla \bm u^{m+1+j-k})+ \nu \delta t\sum_{i,j=1}^{k-1}h_{ij}(\Delta \bm u^{m+2+i-k},\Delta \bm u^{n+2+j-k})\\
&+\rho\nu\delta t\sum_{n=k-1}^m\|\Delta C_k^{\beta_k}(\bm u^{n+1})\|^2 +\nu\kappa_k\delta t\sum_{n=k-1}^m\|\Delta \bm u^{n+1}\|^2\\
\le& C \nu \delta t \sum_{n=k-1}^{m} \|\nabla C_k^{\beta_k}(\bm u^{n})\|^2 +C_I\\
\le&C \nu \delta t \sum_{n=0}^{m} \|\nabla \bm u^n\|^2+C_{II},
\end{split}
\end{equation*}
where $C_I$ is a constant depending on $\|\nabla \bm u^{i}\|^2$ and $\delta t \|\Delta \bm u^{i}\|^2,\,i=0,1,...,k-1$.  By the assumption on the initial $k$ steps, we have that $C_{II}$ only depends on $\bm u^0$. One the other hand, let  $\lambda^g_k$ and $\lambda^h_k$ are the smallest eigenvalues of $G_k=(g_{ij})$ and $H_k=(h_{ij})$ respectively, then we have
\begin{small}
\begin{equation*}
\sum_{i,j=1}^{k}g_{ij}(\nabla \bm u^{m+1+i-k},\nabla \bm u^{m+1+j-k})+ \nu \delta t\sum_{i,j=1}^{k-1}h_{ij}(\Delta \bm u^{m+2+i-k},\Delta \bm u^{n+2+j-k}) \ge \lambda^{g}_k\|\nabla \bm u^{m+1}\|^2+\lambda^h_k \nu \delta t \|\Delta \bm u^{m+1}\|^2.
\end{equation*}
\end{small}
Combining the above two inequalities, we have
\begin{equation*}
\begin{split}
&\lambda^{g}_k\|\nabla \bm u^{m+1}\|^2+\lambda^h_k \nu \delta t \|\Delta \bm u^{m+1}\|^2++\rho\nu\delta t\sum_{n=k-1}^m\|\Delta C_k^{\beta_k}(\bm u^{n+1})\|^2 +\nu\kappa_k\delta t\sum_{n=k-1}^m\|\Delta \bm u^{n+1}\|^2\\
\le& C \nu \delta t \sum_{n=0}^{m} \|\nabla \bm u^n\|^2+C_{II}.
\end{split}
\end{equation*}}We can then obtain the desired bound on the velocity by applying Lemma \ref{Gron2} to the above. Finally the bound on the pressure can be derived  by taking $q=p^{n+1}$ in \eqref{BDF2NS2} and {\color{black} using Lemma \ref{stokespre}}.
\end{proof}

\begin{rem}
	The above theorem provides the first unconditional stability results for any decoupled schemes of third- or higher-order for  time-dependent Stokes equations. It also improves the previous result in \cite{HS2023} for the second-order scheme with $\beta=5$ to $\beta= 3$.
\end{rem}
\section{The BDF-IMEX schemes and  error analysis}
In this section, we construct the $k$-th order  BDF-IMEX schemes  for the Navier-Stokes equations and carry out global-in-time error analysis up to fourth order scheme by induction.

\subsection{A general form of  BDF-IMEX schemes}

 Combining the new BDF type scheme with the consistent splitting schemes in \cite{Shen03JCP}, using the notations introduced in \eqref{ABC_def} and choosing $\beta_k$ as \eqref{betak}, we construct the  $k$-th $(k=2,3,4)$ order schemes for  \eqref{NS} as follows:
 \begin{subequations}\label{scheme_NS}
\begin{align}
&\frac{A_k^{\beta_k}(\bm u^{n+1})}{\delta t}-\nu \Delta B_k^{\beta_k}(\bm u^{n+1})+\nabla C_k^{\beta_k}(p^{n})+C_k^{\beta_k}(\bm u^n)\cdot \nabla C_k^{\beta_k}(\bm u^n)=\bm f^{n+\beta_k}, \label{scheme_NS1}\\
&(\nabla p^{n+1},\nabla q)=(\bm f^{n+1}-\bm u^{n+1} \cdot \nabla \bm u^{n+1}-\nu \nabla\times\nabla\times \bm u^{n+1},\nabla q),\quad\forall q\in H^1(\Omega).\label{scheme_NS2}
\end{align}
\end{subequations}

\subsection{Error analysis}

To simplify the presentation, we take $\nu=1$ in \eqref{scheme_NS1} and denote
\begin{equation*}
t^n=n\,\delta t,\quad  \bm e^n=\bm u^n-\bm u(\cdot, t^n),\quad e_p^n=p^n-p(\cdot,t^n).
\end{equation*}

\begin{theorem}\label{ThmNS}
Let {\color{black} $\Omega \subset \mathbb{R}^d$ satisfies the conditions in Lemma \ref{stokespre},} $d=2,3$, $T>0$, $\bm u_0\in \mathbb{V} \cap \bm H_0^2$ and $\bm u$ be the solution of \eqref{NS}. Assuming that $\|\bm f(\cdot,t)\| \le C_f, \, \forall t \in[0, T]$ and  $\bm u^{i}$ are computed  such that $\|\nabla \bm e^{i}\|^2+\delta t \|\Delta \bm e^{i}\|^2\le C\delta t^{2k} \|\nabla \bm u^0\|^2,\, {\color{black}i=0,..,k-1}$.
Let $\bm u^j \,(j\ge k)$ be the solution of  \eqref{scheme_NS} with {\color{black} $\beta= \beta_k$} chosen as in \eqref{betak}, and
  assume that the exact solutions are sufficiently smooth such that
\begin{equation}\label{NScond2}
\bm u \in L^2(0,T;H^2),\quad \frac{\partial ^{k} \bm u}{\partial t^{k}} \in L^2(0,T;H^2),\quad \frac{\partial ^{k+1} \bm u}{\partial t^{k+1}} \in L^2(0,T;L^2),\quad \frac{\partial ^{k} p}{\partial t^{k}} \in L^2(0,T;H^1).
\end{equation}
Then for  $n+1 \le T/\delta t$ with $\delta t$ sufficiently small, we have
\begin{equation}\label{Thmerror}
\|\nabla \bm e^{n+1}\|^2+\delta t \sum_{i=0}^{n+1}(\|\Delta {\bm e}^i\|^2+\|\nabla e_p^i\|^2 ) \le C\delta t^{2k},
\end{equation}
where the constants  $C$ are dependent on $T,\, \Omega$ and  the exact solution $\bm u$, but are independent of $\delta t$.

\end{theorem}
\begin{proof}
		
Since our focus is on the error analysis for the semi-discrete scheme,  we assume $\bm f^i=\bm f(t^i) \;\forall i,$ and $\bm u^i,\, p^i,\,i\le k-1$ are  computed with proper initialization procedure such that \eqref{Thmerror} holds {\color{black} for $n \le k-1$}.

Firstly, we denote
\begin{equation}\label{C0def}
C_{H^1}:= \mathop{\rm max}\limits_{0\le t \le T} \|\nabla \bm u(\cdot,t)\|\,\, {\rm and} \,\, C_0:=C_{H^1}+1.
\end{equation}
We need to prove a uniform bound of $\|\nabla \bm u^n\|$ by induction,
\begin{equation}\label{C0}
\|\nabla \bm u^i\| \le C_0, \quad  \forall  i\le T/{\delta t},
\end{equation}
 In the following, we shall use $C$ to denote a positive constant independent of  $\delta t$, 
  which can change from one step to another and we use $\varepsilon>0$ to denote a constant which can be arbitrarily small.

Under the assumption,  \eqref{C0} certainly holds for $i=0$. Now
suppose we have
\begin{equation}\label{H1boundpre}
\|\nabla \bm u^i\| \le C_0,\,\, \forall i \le n,
\end{equation}
we shall prove below
\begin{equation}\label{H1bound}
\|\nabla \bm u^{n+1}\| \le C_0,
\end{equation}
for the same constant $C_0$.

\textbf{Step 1: Bounds for $\delta t \sum_{q=0}^i \|\Delta \bm u^q\|^2$,  $\forall i\le n$.}
Considering \eqref{scheme_NS1} at step $i+1 \le n$ and taking the inner product with $-\delta t \Delta C_k^{\beta_k}({\bm u}^{i+1})$. For the first term on the left hand side, it follows from Lemma \ref{Gstability} and Lemma \ref{AC_lemma} that there exists a symmetric positive definite matrix $G_k=(g_{lj})\in \mathbb{R}^{k\times k}$ such that
\begin{equation}\label{AC_error}
\big(A_k^{\beta_k}({\bm u}^{i+1}),-\Delta C_k^{\beta_k}({\bm u}^{i+1}) \big) \ge \sum_{l,j=1}^{k}g_{lj}(\nabla {\bm u}^{i+1+l-k},\nabla {\bm u}^{i+1+j-k})-\sum_{l,j=1}^{k}g_{lj}(\nabla {\bm u}^{i+l-k},\nabla {\bm u}^{i+j-k}).
\end{equation}

For the second term, we split $B_k^{\beta_k}({\bm u}^{i+1})$ as \eqref{splitB} and choose $\eta_k=0.71$, then we have
\begin{equation}\label{B_error}
\begin{split}
\delta t \big(-\Delta B_k^{\beta_k}({\bm u}^{i+1}),-\Delta C_k^{\beta_k}({\bm u}^{i+1}) \big)
& = 0.71\delta t \|\Delta C_k^{\beta_k}({\bm u}^{i+1})\|^2+\delta t\big(\Delta D_k^{\beta_k}({\bm u}^{i+1}),\Delta C_k^{\beta_k}({\bm u}^{i+1}) \big)\\
& + \delta t\big(\Delta F_k^{\beta_k}({\bm u}^{i+1}),\Delta C_k^{\beta_k}({\bm u}^{i+1}) \big),
\end{split}
\end{equation}
and for $\big(\Delta D_k^{\beta_k}({\bm u}^{i+1}),\Delta C_k^{\beta_k}({\bm u}^{i+1}) \big) $, thanks to Lemma \ref{Gstability} and Lemma \ref{DC_lemma}, there exists a symmetric positive definite matrix $H_k=(h_{lj})\in \mathbb{R}^{(k-1)\times (k-1)}$ such that
\begin{equation}\label{DC_error}
\big(\Delta D_k^{\beta_k}({\bm u}^{i+1}),\Delta C_k^{\beta_k}({\bm u}^{i+1}) \big) \ge \sum_{l,j=1}^{k-1}h_{lj}(\Delta {\bm u}^{i+2+l-k},\Delta {\bm u}^{i+2+j-k})-\sum_{l,j=1}^{k-1}h_{lj}(\Delta {\bm u}^{i+1+l-k},\Delta {\bm u}^{i+1+j-k}),
\end{equation}
and for $\big(\Delta F_k^{\beta_k}({\bm u}^{i+1}),\Delta C_k^{\beta_k}({\bm u}^{i+1}) \big)$, \eqref{Fk_ineq} implies
\begin{equation}\label{Fk_error}
\big(\Delta F_k^{\beta_k}({\bm u}^{i+1}),\Delta C_k^{\beta_k}({\bm u}^{i+1}) \big) \ge \kappa_k \|\Delta {\bm u}^{i+1}\|^2+U_k(\Delta {\bm u}^{i+1},...,\Delta {\bm u}^{i+3-k})-U_k(\Delta {\bm u}^{i},...,\Delta {\bm u}^{i+2-k}).
\end{equation}
For the term with $C_k^{\beta_k}(\bm u^i) \cdot \nabla C_k^{\beta_k}(\bm u^i) $, making use of \eqref{eq:ineq2d} and the Poincar\'e type inequality, we have
\begin{equation}\label{NL2D}
\begin{split}
&\big(C_k^{\beta_k}(\bm u^i) \cdot \nabla C_k^{\beta_k}(\bm u^i) , \Delta C_k^{\beta_k}({\bm u}^{i+1}) \big)  \\
\le &\Big| \big(C_k^{\beta_k}({\bm u}^i)\cdot \nabla C_k^{\beta_k}({\bm u}^i), \Delta C_k^{\beta_k}({\bm u}^{i+1}) \big)\Big| \\
\le& c\|C_k^{\beta_k}({\bm u}^i)\|_1 \|C_k^{\beta_k}({\bm u}^i)\|_1^{1/2}\|C_k^{\beta_k}({\bm u}^i)\|_2^{1/2}\|\Delta C_k^{\beta_k}({\bm u}^{i+1})\|\\
\le &C(\varepsilon)\|C_k^{\beta_k}({\bm u}^i)\|_1^2\|C_k^{\beta_k}({\bm u}^i)\|_1\|C_k^{\beta_k}({\bm u}^i)\|_2+\varepsilon\|\Delta C_k^{\beta_k}({\bm u}^{i+1})\|^2\\
\le &C(\varepsilon)\|\nabla C_k^{\beta_k}({\bm u}^i)\|^6+\varepsilon\|\Delta C_k^{\beta_k}({\bm u}^{i})\|^2+\varepsilon\|\Delta C_k^{\beta_k}({\bm u}^{i+1})\|^2,
\end{split}
\end{equation}where we used  $\|C_k^{\beta_k}({\bm u}^i)\|_2^2 \le C\|\Delta C_k^{\beta_k}({\bm u}^i)\|^2$ in the last step.

For the term with $C_k^{\beta_k}(p^i)$, we have
\begin{equation}\label{eq:pterm}
\Big|\big( \nabla C_k^{\beta_k}(p^i), -\Delta C_k^{\beta_k}({\bm u}^{i+1}) \big)\Big| \le \|\nabla C_k^{\beta_k}(p^i)\|\|\Delta C_k^{\beta_k}({\bm u}^{i+1})\|.
\end{equation}

To estimate $\|\nabla C_k^{\beta_k}(p^i)\|$, we  follow a similar procedure as in \cite{liuCPAM}: first rewriting \eqref{scheme_NS2} as
\begin{equation}\label{NSb22}
\big(\nabla p^{i}, \nabla q \big)=\big(\bm f^{i}-{\bm u}^{i} \cdot \nabla {\bm u}^{i}, \nabla q \big)+ \big( \nabla p_s({\bm u}^{i}), \nabla q \big), \quad \forall i \le n,
\end{equation}
where $p_s({\bm u}^{i})$ is the Stokes pressure associated with ${\bm u}^{i}$ and hence
\begin{equation}\label{pressure}
\big(\nabla C_k^{\beta_k}(p^i), \nabla q \big)=\big(C_k^{\beta_k}(\bm f^i)-C_k^{\beta_k}({\bm u}^i\cdot \nabla {\bm u}^i), \nabla q \big)+ \big( \nabla p_s(C_k^{\beta_k}({\bm u}^i)), \nabla q \big).
\end{equation}
Now, taking $q=C_k^{\beta_k}(p^i)$, we have
\begin{equation}
    \|\nabla C_k^{\beta_k}(p^i)\|  \le \|C_k^{\beta_k}(\bm f^i)-C_k^{\beta_k}({\bm u}^i\cdot \nabla {\bm u}^i)\|+\|\nabla p_s(C_k^{\beta_k}({\bm u}^i))\|.
\end{equation}

It follows from  the Sobolev inequality and the elliptic regularity estimate that
\begin{equation}\label{fNL2D}
\begin{split}
 \|C_k^{\beta_k}(\bm f^i)-C_k^{\beta_k}({\bm u}^i\cdot \nabla {\bm u}^i)\|^2
& \le C\sum_{q=0}^{k-1}\|{\bm f}^{i-q}\|^2+C\sum_{q=0}^{k-1}\|{\bm u}^{i-q}\cdot \nabla {\bm u}^{i-q}\|^2\\
& \le C\sum_{q=0}^{k-1}\|{\bm f}^{i-q}\|^2+C\sum_{q=0}^{k-1}\|\nabla {\bm u}^{i-q}\|^3\| \nabla {\bm u}^{i-q}\|_1\\
& \le C\sum_{q=0}^{k-1}\|{\bm f}^{i-q}\|^2+C(\varepsilon)\sum_{q=0}^{k-1}\|\nabla {\bm u}^{i-q}\|^6+\varepsilon\sum_{q=0}^{k-1}\|\Delta {\bm u}^{i-q}\|^2,
\end{split}
\end{equation}
{\color{black} where  we used the following inequality (cf. section 4 in \cite{liuCPAM}),
\begin{equation*}
\|\bm u^i\cdot \nabla \bm u^i\|^2\le \|\bm u^i\|^2_{L^6}\|\nabla \bm u^i\|^2_{L^3} \le C\|\nabla \bm u^i\|^3\|\nabla \bm u^i\|_1, \quad d=2,3.
\end{equation*}}
As a result,  by making use of Lemma \ref{stokespre}, we can estimate \eqref{eq:pterm} as
\begin{equation}
\begin{split}
& \big(\nabla C_k^{\beta_k}(p^i), -\Delta C_k^{\beta_k}({\bm u}^{i+1}) \big) \\
\le &\|\Delta C_k^{\beta_k}({\bm u}^{i+1})\|\big( \|C_k^{\beta_k}(\bm f^i)-C_k^{\beta_k}({\bm u}^i\cdot \nabla {\bm u}^i)\|+\|\nabla p_s(C_k^{\beta_k}({\bm u}^i))\| \big) \\
\le &C(\varepsilon)\|C_k^{\beta_k}(\bm f^i)-C_k^{\beta_k}({\bm u}^i\cdot \nabla {\bm u}^i)\|^2+\varepsilon\|\Delta C_k^{\beta_k}({\bm u}^{i+1})\|^2 +\frac{\gamma}{2}\|\nabla p_s(C_k^{\beta_k}({\bm u}^i))\|^2+\frac{1}{2\gamma}\|\Delta C_k^{\beta_k}({\bm u}^{i+1})\|^2\\
\le & C(\varepsilon) \sum_{q=0}^{k-1}\big(\|{\bm f}^{i-q}\|^2+\|\nabla {\bm u}^{i-q}\|^6 \big)+\varepsilon( \sum_{q=0}^{k-1}\|\Delta {\bm u}^{i-q}\|^2+\|\Delta C_k^{\beta_k}({\bm u}^{i+1})\|^2)\\
& +\gamma(\frac{1}{4}+\frac{\varepsilon}{2})\|\Delta C_k^{\beta_k}({\bm u}^i)\|^2+C(\varepsilon)\|\nabla C_k^{\beta_k}({\bm u}^i)\|^2+\frac{1}{2\gamma}\|\Delta C_k^{\beta_k}({\bm u}^{i+1})\|^2.
\end{split}
\end{equation}
with $\gamma$ can be any positive number.

Finally, for the right hand side of \eqref{scheme_NS1}, we have
\begin{equation}\label{rhs1}
\big( \bm f^{i+\beta_k}, -\Delta C_k^{\beta_k}({\bm u}^{i+1}) \big ) \le C(\varepsilon) \|\bm f^{i+\beta_k}\|^2+\varepsilon\|\Delta C_k^{\beta_k}({\bm u}^{i+1})\|^2.
\end{equation}

Combining \eqref{AC_error} to \eqref{rhs1} and choosing $\gamma=\sqrt{2}$ as before, we obtain
\begin{small}
\begin{equation}\label{tosum}
\begin{split}
& \sum_{l,j=1}^{k}g_{lj}(\nabla {\bm u}^{i+1+l-k},\nabla {\bm u}^{i+1+j-k})-\sum_{l,j=1}^{k}g_{lj}(\nabla {\bm u}^{i+l-k},\nabla {\bm u}^{i+j-k})+\delta t\sum_{l,j=1}^{k-1}h_{lj}(\Delta {\bm u}^{i+2+l-k},\Delta {\bm u}^{i+2+j-k})\\
&-\delta t\sum_{l,j=1}^{k-1}h_{lj}(\Delta {\bm u}^{i+1+l-k},\Delta {\bm u}^{i+1+j-k})+0.71 \delta t\|\Delta C_k^{\beta_k}({\bm u}^{i+1}) \|^2 +\kappa_k \delta t\|\Delta {\bm u}^{i+1}\|^2\\
& +\delta tU_k(\Delta {\bm u}^{i+1},...,\Delta {\bm u}^{i+3-k})-\delta tU_k(\Delta {\bm u}^{i},...,\Delta {\bm u}^{i+2-k})\\
 \le &C(\varepsilon)\delta t \Big(\sum_{q=0}^{k-1}\big(\|{\bm f}^{i-q}\|^2+\|\nabla {\bm u}^{i-q}\|^6 \big)+\|\bm f^{i+\beta_k}\|^2 +\|\nabla C_k^{\beta_k}({\bm u}^i)\|^2+\|\nabla C_k^{\beta_k}({\bm u}^i)\|^6\Big)\\
&+\delta t \varepsilon( \sum_{q=0}^{k-1}\|\Delta {\bm u}^{i-q}\|^2+\|\Delta C_k^{\beta_k}({\bm u}^{i})\|^2+\|\Delta C_k^{\beta_k}({\bm u}^{i+1})\|^2)
+(\frac{\sqrt{2}}{4}+\frac{\sqrt{2}\varepsilon}{2})\delta t \|\Delta C_k^{\beta_k}({\bm u}^i)\|^2\\
&+\frac{\sqrt{2}}{4}\delta t\|\Delta C_k^{\beta_k}({\bm u}^{i+1})\|^2.
\end{split}
\end{equation}
\end{small}


Now, we can choose $\varepsilon$ small enough such that there exists $\rho>0$ such that
\begin{equation}\label{epsilon}
0.71  -( \frac{\sqrt{2}}{2}+\frac{\sqrt{2}\varepsilon}{2}+2\varepsilon) \ge \rho>0 \quad \text{and} \quad \kappa_k- k\varepsilon \ge \rho>0.
\end{equation}
Then taking the sum on \eqref{tosum} for $i$ from $k-1$ to $m-1$ with $m \le n$ and dropping some unnecessary terms, we can obtain:
\begin{small}
\begin{equation}\label{eq:sum}
\begin{split}
& \lambda_k^g\|\nabla {\bm u}^{m}\|^2+\rho \delta t \sum_{i=k}^{m}(\|\Delta C_k^{\beta_k}({\bm u}^i)\|^2+\|\Delta {\bm u}^{i}\|^2) \\
\le &C\delta t \sum_{i=k-1}^{m-1}\big(\|\nabla {\bm u}^i\|^6+\|\nabla C_{k}^{\beta_k}({\bm u}^i)\|^{\color{black}6}\big)+C\delta t \sum_{i=k-1}^{m-1}\|\nabla C_k^{\beta_k}({\bm u}^i)\|^2 +C\delta t\sum_{i=k-1}^{m-1} (\|\bm f^i\|^2+\|\bm f^{i+\beta_k}\|^2)+M_0\\
\le & C\delta t \sum_{i=0}^{m-1}\|\nabla {\bm u}^i\|^6+C\delta t \sum_{i=0}^{m-1}\|\nabla {\bm u}^i\|^2+CTC_f^2+M_0,\quad \forall m \le n,
\end{split}
\end{equation}
\end{small}where $\lambda_k^g>0$ is the smallest eigenvalue of $G_k=(g_{lj})$, $M_0$ is a constant only depends on the data from initial $k-1$ steps and we used $\|\bm f(\cdot,t)\| \le C_f, \, \forall t \in[0, T]$. Next, noting that $\|\nabla {\bm u}^i\| \le C_0,\,\forall i \le n$ under the induction assumption {\color{black} and $C_0>1$ from \eqref{C0def}}, we can obtain from \eqref{eq:sum}:
\begin{equation}\label{H2bound}
\|\nabla {\bm u}^{m}\|^2+\delta t \sum_{i=k}^{m}(\|\Delta C_k^{\beta_k}({\bm u}^i)\|^2+\|\Delta {\bm u}^{i}\|^2) \le CT(C_0^6+C_f^2)+M_0, \quad \forall m\le n.
\end{equation}

\textbf{Step 2: Error estimate for $\|\nabla {\bm e}^{n+1}\|$.}
From \eqref{NS} and \eqref{scheme_NS}, we can write down the error equation for ${\bm u}^{i+1}$ and $p^{i+1}$ as
\begin{equation}\label{NSerror}
\begin{split}
& A_k^{\beta_k}({\bm e}^{i+1})-\delta t\Delta B_k^{\beta_k}({\bm e}^{i+1})+\delta t(C_k^{\beta_k}({\bm u}^i) \cdot \nabla C_k^{\beta_k}({\bm u}^i)-C_k^{\beta_k}[{\bm u}(t^{i})]\cdot \nabla C_k^{\beta_k}[{\bm u}(t^{i})])+\delta t\nabla C_k^{\beta_k}({e}_p^i)\\
&=\delta t P_k^i+\delta tQ_k^i+R_k^i+\delta tS_k^i,
\end{split}
\end{equation}
where $P_k^i$, $Q_k^i$, $R_k^i$ $S_k^i$ are given by
\begin{equation}\label{P}
P_k^i= \nabla p(t^{i+\beta_k})-\nabla C_k^{\beta_k}(p(t^i))=\frac{1}{(k-1)!}\sum_{q=0}^{k-1}c_{k,q}(\beta_k)\int_{t^{i+1+q-k}}^{t^{i+\beta_k}}(t^{i+1+q-k}-s)^{k-1}\nabla \frac{\partial^k p}{\partial t^k}(s)ds,
\end{equation}
with $c_{k,q}(\beta_k)$ defined in \eqref{TaylorC} and
\begin{equation}\label{Q}
Q_k^i=- \Delta \bm u (t^{i+\beta_k})+\Delta B_k^{\beta_k} (\bm u(t^{i+1}))
=\frac{-1}{(k-1)!}\sum_{q=0}^{k-1}b_{k,q}(\beta_k)\int_{t^{i+2+q-k}}^{t^{i+\beta_k}}(t^{i+2+q-k}-s)^{k-1}\Delta  \frac{\partial^k \bm u}{\partial t^k}(s)ds,
\end{equation}
\begin{equation}\label{R}
\begin{split}
R_k^i=\delta t \bm u_t(t^{i+\beta_k})-A_k^{\beta_k}(\bm u(t^{i+1}))=\frac{1}{k!}\sum_{q=0}^{k}a_{k,q}(\beta_k)\int_{t^{i+1+q-k}}^{t^{i+\beta_k}}(t^{i+1+q-k}-s)^{k} \frac{\partial^{k+1} \bm u}{\partial t^{k+1}}(s)ds,
\end{split}
\end{equation}
and
\begin{equation}\label{S}
\begin{split}
S_k^i& =\bm u(t^{i+\beta_{k}})\cdot \nabla \bm u(t^{i+\beta_k})-C_k^{\beta_k}[\bm u(t^i)] \cdot \nabla C_k^{\beta_k}[\bm u(t^i)]\\
&=\bm u(t^{i+\beta_k})\cdot \nabla (\bm u(t^{i+\beta_k})-C_k^{\beta_k}[{\bm u}(t^i)])-(C_k^{\beta_k}[{\bm u}(t^i)]-\bm u(t^{i+\beta_k})) \cdot \nabla C_k^{\beta_k}[\bm u(t^i)].
\end{split}
\end{equation}
Next, we take the inner product of  \eqref{NSerror} with $-\Delta C_k^{\beta_k}({\bm e}^{i+1})$. For the first term on the left hand side, same as \eqref{AC_error}, we have
\begin{equation}\label{AC_error2}
\big(A_k^{\beta_k}({\bm e}^{i+1}),-\Delta C_k^{\beta_k}({\bm e}^{i+1}) \big) \ge \sum_{l,j=1}^{k}g_{lj}(\nabla {\bm e}^{i+1+l-k},\nabla {\bm e}^{i+1+j-k})-\sum_{l,j=1}^{k}g_{lj}(\nabla {\bm e}^{i+l-k},\nabla {\bm e}^{i+j-k}).
\end{equation}
We handle the term  with $B_k^{\beta_k}({\bm e}^{i+1})$ similarly  as in  \eqref{B_error}-\eqref{Fk_error} to obtain,
\begin{equation}\label{BC_error}
\begin{split}
&\delta t\big(-\Delta B_k^{\beta_k}({\bm e}^{i+1}) ,-\Delta C_k^{\beta_k}({\bm e}^{i+1}) \big) \\
 \ge & 0.71\delta t\|\Delta C_k^{\beta_k}({\bm e}^{i+1})\|{\color{black}^2}+\kappa_k \delta t\|\Delta {\bm e}^{i+1}\|^2 + \delta t\sum_{l,j=1}^{k-1}h_{lj}(\Delta {\bm e}^{i+2+l-k},\Delta {\bm e}^{i+2+j-k})\\
&-\delta t\sum_{l,j=1}^{k-1}h_{lj}(\Delta {\bm e}^{i+1+l-k},\Delta {\bm e}^{i+1+j-k})+\delta tU_k(\Delta {\bm e}^{i+1},...,\Delta {\bm e}^{i+3-k})-\delta t U_k(\Delta {\bm e}^{i},...,\Delta {\bm e}^{i+2-k}).
\end{split}
\end{equation}
For the third term on the left hand side of \eqref{NSerror}, we rewrite it as
\begin{equation}\label{eq:split}
\begin{split}
& C_k^{\beta_k}({\bm u}^i) \cdot \nabla C_k^{\beta_k}({\bm u}^i)-C_k^{\beta_k}[{\bm u}(t^{i})]\cdot \nabla C_k^{\beta_k}[{\bm u}(t^{i})] \\ =&C_k^{\beta_k}({\bm u}^i) \cdot \nabla C_k^{\beta_k}({\bm u}^i)-C_k^{\beta_k}[{\bm u}(t^{i})] \cdot \nabla C_k^{\beta_k}({\bm u}^i)+C_k^{\beta_k}[{\bm u}(t^{i})]\cdot \nabla C_k^{\beta_k}({\bm u}^i)-C_k^{\beta_k}[{\bm u}(t^{i})]\cdot \nabla C_k^{\beta_k}[{\bm u}(t^{i})]\\
=& C_k^{\beta_k}({\bm e}^i) \cdot \nabla C_k^{\beta_k}({\bm u}^i)+ C_k^{\beta_k}[{\bm u}(t^{i})] \cdot \nabla C_k^{\beta_k}({\bm e}^i).
\end{split}
\end{equation}
Therefore, it follows from \eqref{eq:ineq2} that
\begin{equation}
\begin{split}
&\big(C_k^{\beta_k}({\bm u}^i) \cdot \nabla C_k^{\beta_k}({\bm u}^i)-C_k^{\beta_k}[{\bm u}(t^{i})]\cdot \nabla C_k^{\beta_k}[{\bm u}(t^{i})], -\Delta C_k^{\beta_k}({\bm e}^{i+1}) \big)\\
=&\big(C_k^{\beta_k}({\bm e}^i) \cdot \nabla C_k^{\beta_k}({\bm u}^i),  -\Delta C_k^{\beta_k}({\bm e}^{i+1})  \big)+\big(C_k^{\beta_k}[{\bm u}(t^{i})] \cdot \nabla C_k^{\beta_k}({\bm e}^i), -\Delta C_k^{\beta_k}({\bm e}^{i+1})  \big)\\
 \le& C\|\nabla C_k^{\beta_k}({\bm e}^i)\|\|C_k^{\beta_k}({\bm u}^i)\|_2\|\Delta C_k^{\beta_k}({\bm e}^{i+1})\|+C\|C_k^{\beta_k}[{\bm u}(t^{i})]\|_2\|\nabla C_k^{\beta_k}({\bm e}^i)\|\|\Delta C_k^{\beta_k}({\bm e}^{i+1})\|\\
\le& C(\varepsilon)\|\nabla C_k^{\beta_k}({\bm e}^i)\|^2\|\Delta C_k^{\beta_k}({\bm u}^i)\|^2+C(\varepsilon)\|C_k^{\beta_k}[{\bm u}(t^{i})]\|_2^2\|\nabla C_k^{\beta_k}({\bm e}^i)\|^2+\varepsilon\|\Delta C_k^{\beta_k}({\bm e}^{i+1})\|^2.\\
\end{split}
\end{equation}

For the term with $C_k^{\beta_k}({e}_p^i)$, we have
\begin{equation}\label{ep1}
\big(\nabla C_k^{\beta_k}({e}_p^i), -\Delta  C_k^{\beta_k}({\bm e}^{i+1}) \big) \le  \|\nabla C_k^{\beta_k}({e}_p^i)\|\|\Delta  C_k^{\beta_k}({\bm e}^{i+1})\|.
\end{equation}
To estimate $\|\nabla C_k^{\beta_k}({e}_p^i)\|$, same as in the last step, we make use of the Stokes pressure. First, from \eqref{scheme_NS2}, the error equation for $e_p^i$ can be rewritten as
\begin{equation}\label{errorp0}
(\nabla {e}_p^{i}, \nabla q)=\big({\bm u}(t^{i}) \cdot \nabla {\bm u}(t^{i})-{\bm u}^{i}\cdot \nabla {\bm u}^{i}, \nabla q \big)+ (\nabla p_s( {\bm e}^{i}), \nabla q),
\end{equation}
and hence,
\begin{equation}\label{errorp}
(\nabla C_k^{\beta_k}({e}_p^i), \nabla q)=\Big(C_k^{\beta_k}\big({\bm u}(t^{i}) \cdot \nabla {\bm u}(t^{i})-{\bm u}^{i}\cdot \nabla {\bm u}^{i}\big), \nabla q\Big)+ (\nabla p_s(C_k^{\beta_k}({\bm e}^{i})), \nabla q),
\end{equation}
where $p_s(C_k^{\beta_k}({\bm e}^{i}))$ is the Stokes pressure associated with $C_k^{\beta_k}({\bm e}^{i})$. We let $q=C_k^{\beta_k}({e}_p^i)$ in the above to obtain
\begin{equation}\label{ep}
\|\nabla C_k^{\beta_k}({e}_p^i)\| \le \|C_k^{\beta_k}({\bm u}(t^{i}) \cdot \nabla {\bm u}(t^{i})-{\bm u}^{i}\cdot \nabla {\bm u}^{i})\|+\|\nabla p_s(C_k^{\beta_k}({\bm e}^{i}))\|.
\end{equation}
Similarly as in \eqref{eq:split}, we  rewrite
\begin{equation}
{\bm u}(t^{i}) \cdot \nabla {\bm u}(t^{i})-{\bm u}^{i}\cdot \nabla {\bm u}^{i}=- {\bm e}^i \cdot \nabla {\bm u}^i-\bm u(t^i) \cdot \nabla {\bm e}^i,
\end{equation}
then it follows from the Sobolev inequality and the Poincar\'e type inequality that
\begin{equation}\label{pnon1}
 \|C_k^{\beta_k}({\bm u}(t^{i}) \cdot \nabla {\bm u}(t^{i})-{\bm u}^{i}\cdot \nabla {\bm u}^{i})\|^2 \le C \sum_{q=0}^{k-1}(\|\nabla {\bm e}^{i-q}\|^2\|\Delta {\bm u}^{i-q}\|^2+\|\bm u(t^{i-q})\|_2^2\|\nabla {\bm e}^{i-q}\|^2).
\end{equation}
Now, combining \eqref{ep1} to \eqref{pnon1} and making use of Lemma \ref{stokespre} for the Stokes pressure, we can bound the term with $C_k^{\beta_k}({e}_p^i)$ as
\begin{equation}\label{ptermbound}
\begin{split}
& \big(\nabla C_k^{\beta_k}({e}_p^i), -\Delta  C_k^{\beta_k}({\bm e}^{i+1}) \big) \\
\le &\|C_k^{\beta_k}\big({\bm u}(t^{i}) \cdot \nabla {\bm u}(t^{i})-{\bm u}^{i}\cdot \nabla {\bm u}^{i}\big)\|\|\Delta  C_k^{\beta_k}({\bm e}^{i+1})\|+  \|\nabla p_s(C_k^{\beta_k}({\bm e}^{i}))\|\|\Delta  C_k^{\beta_k}({\bm e}^{i+1})\| \\
\le& C(\varepsilon)\|C_k^{\beta_k}({\bm u}(t^{i}) \cdot \nabla {\bm u}(t^{i})-{\bm u}^{i}\cdot \nabla {\bm u}^{i})\|^2+\varepsilon\|\Delta C_k^{\beta_k}({\bm e}^{i+1})\|^2+ \frac{\gamma}{2}\|\nabla p_s(C_k^{\beta_k}({\bm e}^{i}))\|^2+\frac{1}{2\gamma}\|\Delta  C_k^{\beta_k}({\bm e}^{i+1})\|^2\\
 \le& C(\varepsilon)\sum_{q=0}^{k-1}\|\nabla {\bm e}^{i-q}\|^2(\|\Delta {\bm u}^{i-q}\|^2+\|\bm u(t^{i-q})\|_2^2)
 +(\varepsilon+\frac{1}{2\gamma})\|\Delta C_k^{\beta_k}({\bm e}^{i+1})\|^2+\gamma(\frac{1}{4}+\frac{\varepsilon}{2})\|\Delta C_k^{\beta_k}({\bm e}^{i})\|^2\\
 &+C(\varepsilon)\|\nabla C_k^{\beta_k}({\bm e}^{i})\|^2.
\end{split}
\end{equation}
For the right hand side of \eqref{NSerror}, we derive from \eqref{P}-\eqref{R} that
\begin{equation}\label{Pbound}
\begin{split}
(P_k^i, -\Delta  C_k^{\beta_k}({\bm e}^{i+1})) & \le C(\varepsilon)\|P_k^i\|^2+\varepsilon\|\Delta  C_k^{\beta_k}({\bm e}^{i+1})\|^2 \\
& \le C(\varepsilon)\delta t^{2k-1} \int_{t^{i+1-k}}^{t^{i+\beta_k}}\Big{\|}\nabla \frac{\partial ^k p}{\partial t^k}(s)\Big{\|}^2 ds+\varepsilon\|\Delta  C_k^{\beta_k}({\bm e}^{i+1})\|^2,
\end{split}
\end{equation}
and similarly,
\begin{equation}\label{Qbound}
(Q_k^i, -\Delta  C_k^{\beta_k}({\bm e}^{i+1})) \le C(\varepsilon)\delta t^{2k-1}\int_{t^{i+2-k}}^{t^{i+\beta_k}}\Big{\|}\Delta \frac{\partial ^k \bm u}{\partial t^k}(s)\Big{\|}^2 ds+\varepsilon\|\Delta C_k^{\beta_k}({\bm e}^{i+1})\|^2,
\end{equation}
\begin{equation}\label{Rbound}
\begin{split}
(R_k^i, -\Delta  C_k^{\beta_k}({\bm e}^{i+1})) & \le \frac{C(\varepsilon)}{\delta t}\|R_k^i\|^2+\varepsilon \delta t\| \Delta C_k^{\beta_k}({\bm e}^{i+1})\|^2 \\
&\le  C(\varepsilon)\delta t^{2k} \int_{t^{i+1-k}}^{t^{i+\beta_k}}\Big{\|}\frac{\partial ^{k+1} \bm u}{\partial t^{k+1}}\Big{\|}^2 ds+\varepsilon\delta t \|\Delta C_k^{\beta_k}({\bm e}^{i+1})\|^2.
\end{split}
\end{equation}
For the term with $S_k^i$, it follows from \eqref{eq:ineq2} and \eqref{S} that
\begin{equation}\label{Sbound}
\begin{split}
&(S_k^i, -\Delta  C_k^{\beta_k}({\bm e}^{i+1})) \\
 \le &C \|\bm u(t^{i+\beta_k})\|_2\|\nabla(\bm u(t^{i+\beta_k})-C_k^{\beta_k}[{\bm u}(t^i)]\|\|\Delta C_k^{\beta_k}({\bm e}^{i+1})\|\\
& + C \|C_k^{\beta_k}[{\bm u}(t^{i})]\|_2\|\nabla(\bm u(t^{i+\beta_k})-C_k^{\beta_k}[\bm u(t^i)])\|\|\Delta C_k^{\beta_k}({\bm e}^{i+1})\|\\
\le& C(\varepsilon)(\|\bm u(t^{i+\beta_k})\|_2^2+\|C_k^{\beta_k}[\bm u(t^i)]\|_2^2)\|\nabla(\bm u(t^{i+\beta_k})-C_k^{\beta_k}[\bm u(t^i)]\|^2+\varepsilon \|\Delta C_k^{\beta_k}({\bm e}^{i+1})\|^2\\
\le &C(\varepsilon)\delta t^{2k-1}\int_{t^{i+1-k}}^{t^{i+\beta_k}}\Big{\|} \nabla \frac{\partial ^k \bm u}{\partial t^k}(s)\Big{\|}^2 ds+\varepsilon \|\Delta C_k^{\beta_k}({\bm e}^{i+1})\|^2.
\end{split}
\end{equation}
Now, we combine \eqref{AC_error2} to \eqref{Sbound}, choose $\gamma=\sqrt{2}$ and drop some unnecessary terms to obtain
\begin{small}
\begin{equation}\label{tosum2}
\begin{split}
& \sum_{l,j=1}^{k}g_{lj}(\nabla {\bm e}^{i+1+l-k},\nabla {\bm e}^{i+1+j-k})-\sum_{l,j=1}^{k}g_{lj}(\nabla {\bm e}^{i+l-k},\nabla {\bm e}^{i+j-k})+\delta t\sum_{l,j=1}^{k-1}h_{lj}(\Delta {\bm e}^{i+2+l-k},\Delta {\bm e}^{i+2+j-k})\\
&-\delta t\sum_{l,j=1}^{k-1}h_{lj}(\Delta {\bm e}^{i+1+l-k},\Delta {\bm e}^{i+1+j-k})+0.71 \delta t\|\Delta C_k^{\beta_k}({\bm e}^{i+1}) \|^2 +\kappa_k \delta t\|\Delta {\bm e}^{n+1}\|^2\\
& +\delta tU_k(\Delta {\bm e}^{i+1},...,\Delta {\bm e}^{i+3-k})-\delta tU_k(\Delta {\bm e}^{i},...,\Delta {\bm e}^{i+2-k})\\
\le &C(\varepsilon)\delta t\|\nabla C_k^{\beta_k}({\bm e}^i)\|^2(\|\Delta C_k^{\beta_k}({\bm u}^i)\|^2+\|C_k^{\beta_k}[{\bm u} (t^i)]\|_2^2+1)+(\varepsilon+\frac{\sqrt{2}}{4})\delta t\|\Delta C_k^{\beta_k}({\bm e}^{i+1})\|^2\\
&+(\frac{\sqrt{2}}{4}+\frac{\sqrt{2}\varepsilon}{2})\delta t\|\Delta C_k^{\beta_k}({\bm e}^i)\|^2
+C(\varepsilon)\delta t\sum_{q=0}^{k-1}\|\nabla {\bm e}^{i-q}\|^2(\|\Delta {\bm u}^{i-q}\|^2+\|\bm u(t^{i-q})\|_2^2)\\
& + C(\varepsilon)\delta t^{2k} \int_{t^{i+1-k}}^{t^{i+\beta_k}}\Big(\Big{\|}\nabla \frac{\partial ^k p}{\partial t^k}(s)\Big{\|}^2+\Big{\|}\Delta \frac{\partial ^k \bm u}{\partial t^k}(s)\Big{\|}^2+\Big{\|}\frac{\partial ^{k+1} \bm u}{\partial t^{k+1}}\Big{\|}^2+\Big{\|} \nabla \frac{\partial ^k \bm u}{\partial t^k}(s)\Big{\|}^2\Big)ds.
\end{split}
\end{equation}
\end{small}
Thanks to \eqref{epsilon}, we have
\begin{equation}
0.71-\frac{\sqrt{2}}{2}-\varepsilon-\frac{\sqrt{2}\varepsilon}{2}  >0.
\end{equation}
Taking the sum of \eqref{tosum2} for $i$ from $k-1$ to $n$.  Under the assumption \eqref{NScond2} on the exact solution and the initial steps $\bm u^i,\, \forall i \le k-1$, we can obtain the following after dropping some unnecessary terms:
\begin{small}
\begin{equation}\label{eq:sum2}
\begin{split}
&\lambda_k^g\|\nabla {\bm e}^{n+1}\|^2+\kappa_k\delta t \sum_{i=k}^{n+1}\|\Delta {\bm e}^{i}\|^2\\
 \le & C\delta t \sum_{i=k-1}^{n}\|\nabla {\bm e}^i\|^2(\|\Delta {\bm u}^i\|^2+\|{\bm u}(t^i)\|_2^2)+C\delta t \sum_{i=k-1}^{n}\|\nabla C_k^{\beta_k}({\bm e}^i)\|^2(\|\Delta C_k^{\beta_k}({\bm u}^i)\|^2+\|C_k^{\beta_k}[{\bm u}(t^i)]\|_2^2+1) +CT\delta t^{2k}\\
\le& C\delta t\sum_{i=k-1}^{n}\|\nabla {\bm e}^i\|^2 \Big(\|\Delta {\bm u}^i\|^2+\|{\bm u}(t^i)\|_2^2+\sum_{q=0}^{{\rm min}\{k-1,n-i\}}(\|\Delta C_k^{\beta_k}({\bm u}^{i+q})\|^2+\|C_k^{\beta_k}[{\bm u}(t^{i+q})]\|_2^2+1) \Big)+CT\delta t^{2k}
\end{split}
\end{equation}
\end{small}where $\lambda_k^g>0$ is the smallest eigenvalue of $G_k=(g_{lj})$.
We can then derive from \eqref{H2bound} and assumptions on the exact solution that there exists $C_1>1$, which is independent of $C_0$ and $\delta t$ such that
\begin{subequations}\label{H2bound2}
\begin{align}
\delta t \sum_{i=0}^{m}\|\Delta {\bm u}^i\|^2,\,\delta t \sum_{i=k-1}^{m}\|\Delta C_k^{\beta_k}({\bm u}^i)\|^2 & \le  C_1(C_0^6+1),\, \forall m \le n;\\
\|C_k^{\beta_k}[{\bm u}(t^i)]\|_2^2,\,\|{\bm u}(t^i)\|_2^2 &\le C_1,\, \forall t\le T.
\end{align}
\end{subequations}
Then noting that the assumption on the initial steps $\bm u^i,\,  i=0,..., k-1$ and applying the Gronwall Lemma \ref{Gron2} on \eqref{eq:sum2}, we obtain
\begin{equation}\label{ebarerror}
\|\nabla {\bm e}^{n+1}\|^2+\delta t \sum_{i=0}^{n+1}\|\Delta {\bm e}^{i}\|^2 \le CT\delta t^{2k}\exp(CC_1(C_0^6+1)+T)=: C_{\bm u}^2 \delta t^{2k},
\end{equation}
where $C_{\bm u}$ is a constant independent of $\delta t$. Moreover, since
\begin{equation}
\|\nabla \bm u^{n+1}\| \le \|\nabla \bm u(\cdot, t^{n+1})\|+\|\nabla {\bm e}^{n+1}\|,
\end{equation}
it follows from the definition of $C_0$ in \eqref{C0def} that \eqref{H1bound} is obviously true if we choose
\begin{equation}
\delta t \le {\rm min}\{1, \frac{1}{C_{\bm u}}\},
\end{equation}
and hence the induction process is completed.

\textbf{Step 3: Error estimate for the pressure}. Let $q=e_p^i$ in \eqref{errorp0}, by using  Lemma \ref{stokespre} and the Sobolev inequality, we have
\begin{equation}\label{errorp2}
\begin{split}
\|\nabla e_p^i \|^2 & \le 2\|\bm u(t^i) \cdot \nabla \bm u(t^i)-{\bm u}^i\cdot \nabla {\bm u}^i\|^2+2\|\nabla p_s({\bm e}^i)\|^2\\
& \le C\|\nabla {\bm e}^i\|^2(\|\Delta {\bm u}^i\|^2+\|\bm u(t^i)\|_2^2)+2\|\Delta {\bm e}^i\|^2+C\|\nabla {\bm e}^i\|^2.
\end{split}
\end{equation}
Now, take the sum on \eqref{errorp2}, then \eqref{H2bound2} and \eqref{ebarerror} together imply
\begin{equation}\label{errorp3}
\delta t \sum_{i=0}^{n+1} \|\nabla e_p^i\|^2 \le C\delta t^{2k}.
\end{equation}
Finally, we complete the proof by combining \eqref{ebarerror} and \eqref{errorp3}.

 \end{proof}

\section{Numerical validation and concluding remarks}
We  provide two numerical examples to show that  (i) the higher-order consistent splitting  schemes based on the usual BDF are not unconditionally stable but the new schemes with suitable $\beta$ are, and (ii)  the new schemes with suitable $\beta$ achieve the expected  convergence rates,  followed by some concluding remarks.
\subsection{Numerical results}

\textit{Example 1.} In the first example, we first consider the stokes problem (in the absence of $\bm f$ and the nonlinear term in \eqref{NS}) in $\Omega=(-1,1)\times (-1,1)$ with no-slip boundary condition, and the initial conditions are given as
\begin{subequations}\label{initial}
	\begin{align}
		& u_1(x,y,0)=\sin(2\pi y)\sin^2(\pi x);\\
		& u_2(x,y,0)=-\sin(2\pi x)\sin^2(\pi y).
	\end{align}
\end{subequations}
We set $\nu=0.005$ and use the third- and fourth- order version of \eqref{scheme_stoke}. We use the Legendre-Galerkin method  \cite{STWbook} with $Nx=Ny=128$ modes in space. In Figure \ref{fig:stoke}, we plot the energy evolution obtained from the third- and fourth- order schemes. In both cases, we  observe that the  high-order schemes based on the usual BDF (with $\beta=1$) are unstable even with an extremely small time step ($\delta t=0.0005$ for the third order scheme and $\delta t=0.0002$ for the fourth order scheme), while we can obtain  correct solutions with large time step $\delta t=0.05$ by choosing suitable $\beta$ as specified in previous sections.  We also observe that  these schemes are still stable with  $\delta t=1 $ although the solutions are no longer correct with such large time step.
\begin{figure}[!htbp]
	\centering
	\subfigure{ \includegraphics[scale=.55]{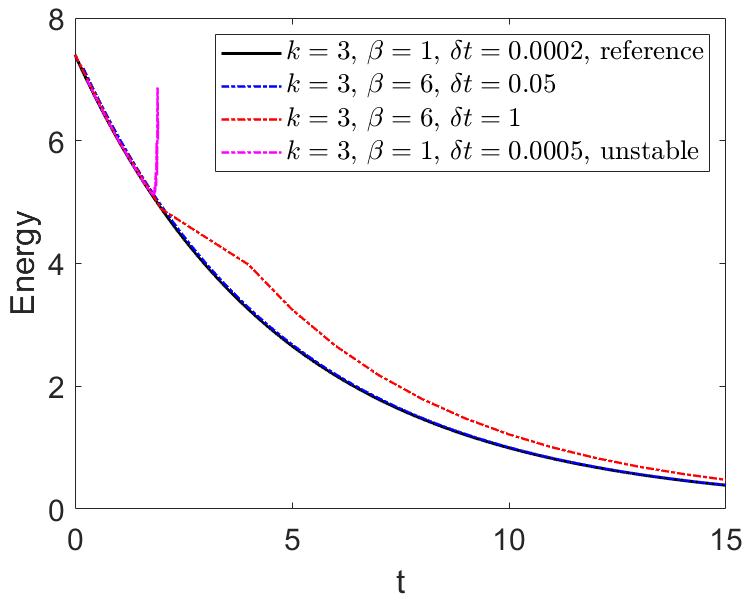}}
	\subfigure{ \includegraphics[scale=.55]{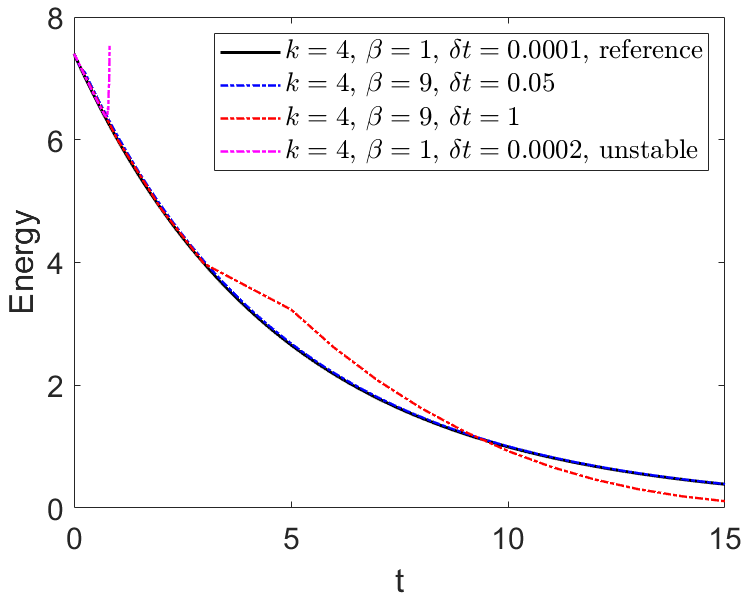}}
	\caption{ Energy evolution for the stokes problem. Left: third order scheme; Right: fourth order scheme. }\label{fig:stoke}
\end{figure}

Next, we consider the Navier-Stokes equation \eqref{NS} with $\nu=0.005$ and the initial conditions are still chosen as \eqref{initial}. We adopt the third- and fourth- order version of \eqref{scheme_NS} and we use the Spectral-Galerkin method with $Nx=Ny=128$ modes in space. In Figure \ref{fig:NS}, we plot the energy evolution obtained from the third- and fourth- order schemes, the reference solution is generated by the fourth-order scheme with $\beta=9,\, Nx=Ny=192,\,\delta t=0.0002$. We observe from Figure that with the same time step $\delta t=0.0005$, the usual BDF3 and BDF4 schemes (with $\beta=1$) are unstable. On the other hand, we can obtain stable and correct solutions using the new third-order (resp. fourth-order) schemes with  $\beta=6$ (resp. $\beta=9$). In Figure \ref{fig:NS}, we plot some snapshots of the vorticity contours at different times.

\begin{figure}[!htbp]
 \centering
 \subfigure[Energy evolution]{ \includegraphics[scale=.5]{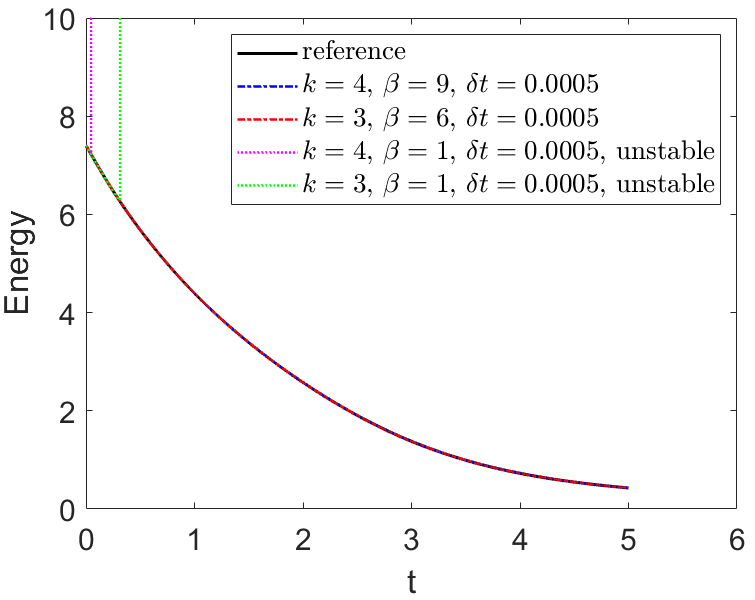}}
  \subfigure[T=0.01]{ \includegraphics[scale=.5]{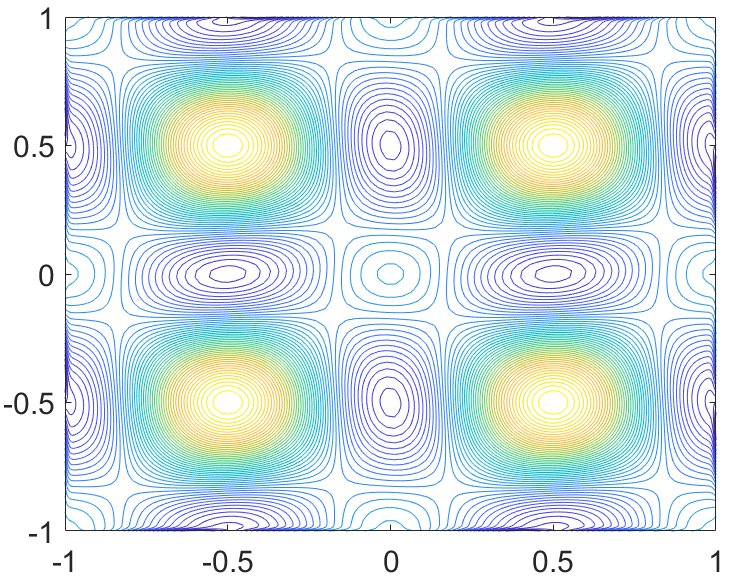}}\\
   \subfigure[T=3]{ \includegraphics[scale=.5]{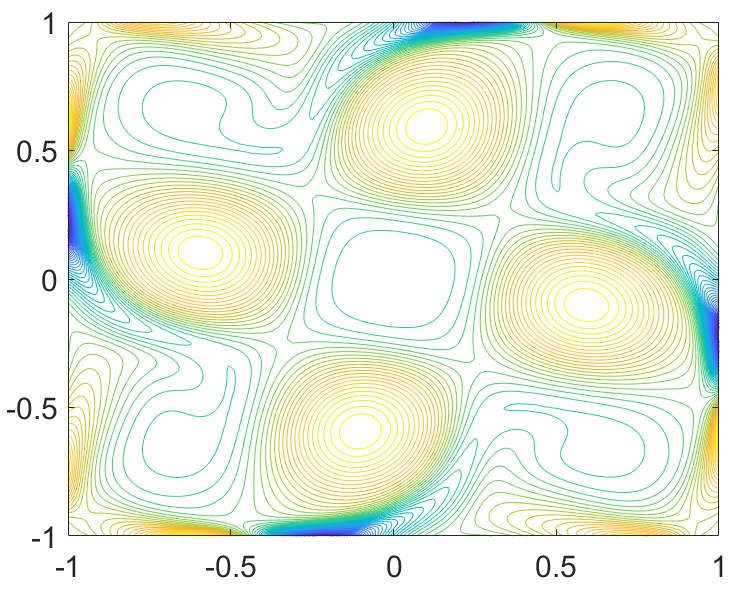}}
  \subfigure[T=5]{ \includegraphics[scale=.5]{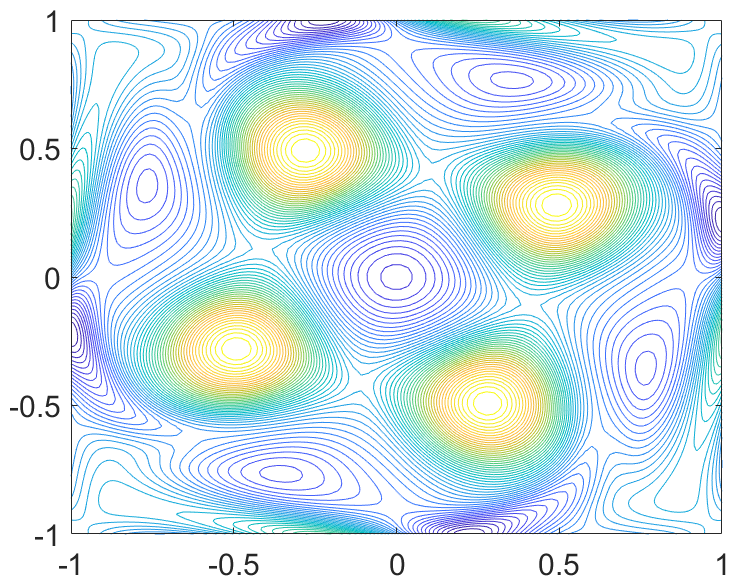}}\\
  \caption{ Energy evolution for the Navier-Stokes equations and snapshots of the vorticity contours at T=0.01, 3, 5.}\label{fig:NS}
 \end{figure}

 \textit{Example 2.} In the second example, we validate the convergence order of the new schemes. Consider the Navier-Stokes equations \eqref{NS} in $\Omega=(-1,1)\times (-1,1)$ with the exact solutions given by
 \begin{align*}
 	& u_1(x,y,t)=\sin(2\pi y)\sin^2(\pi x)\sin(t);\\
 	& u_2(x,y,t)=-\sin(2\pi x)\sin^2(\pi y)\sin(t);\\
 	& p(x,y,t)=\cos(\pi x)\sin(\pi y)\sin(t).
 \end{align*}
 We set $\nu=1$ in \eqref{NS1}, and use the Legendre-Galerkin  method with $Nx=Ny=32$ modes in space so that the spatial discretization error is negligible compared with the time discretization error.
 In Figure \ref{fig:convergence}, we  plot the  convergence rate of the $L^2$ error for the  velocity $error_{\bm u}$,  the $L^2$ error for the pressure $error_p$ and the value of $\|\nabla \cdot \bm u\|$ at $T=1$  by using the $k$-th ($k=2,3,4$)  order schemes \eqref{scheme_NS} with  $\beta_2=3,\,\beta_3=6,\beta_4=9$. We   observe  that the expected convergence rates are achieved  in all test cases.

 \begin{figure}[!htbp]
 	\centering
 	\subfigure{ \includegraphics[scale=.35]{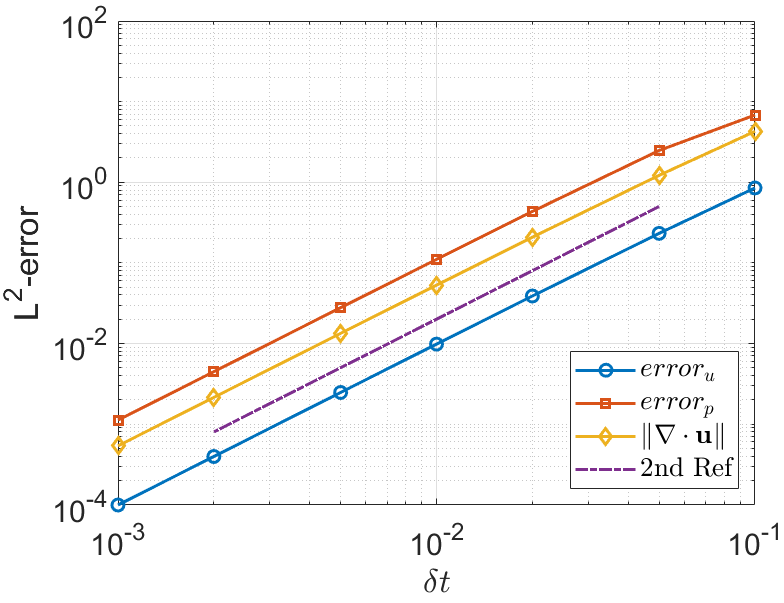}}
 	\subfigure{ \includegraphics[scale=.35]{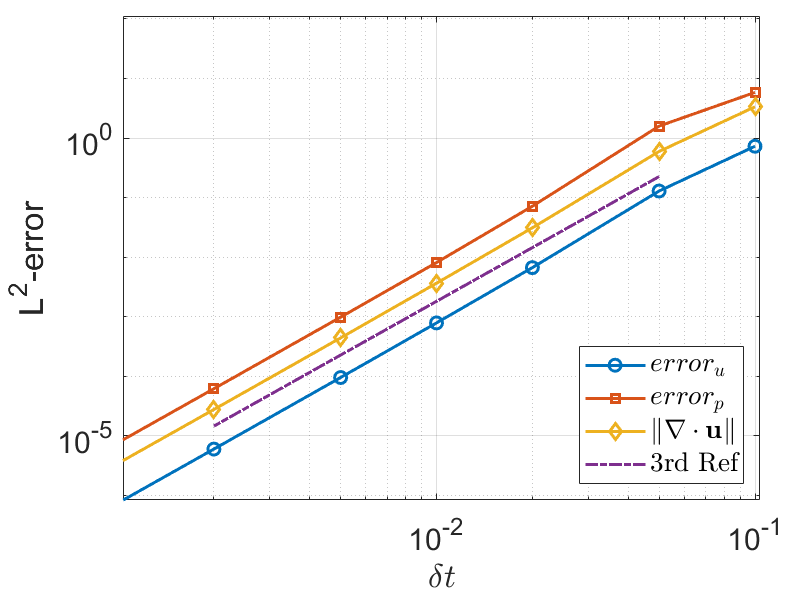}}
 	\subfigure{ \includegraphics[scale=.35]{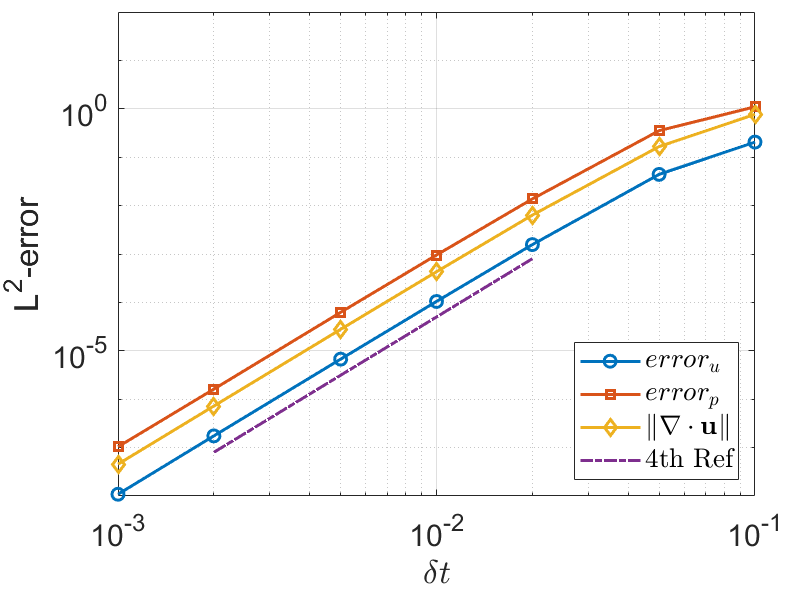}}
 	\caption{ Convergence test for the general BDF type methods. From left to right: second order, third order and fourth order schemes with  $\beta_2=3,\,\beta_3=6,\beta_4=9$. }\label{fig:convergence}
 \end{figure}

\subsection{Concluding remarks}
We considered in this paper the construction and analysis of semi-discrete higher-order consistent splitting schemes for the Navier-stokes equations.
 We constructed  schemes  based on the Taylor expansion at $t^{n+\beta}$ with $\beta\ge 1$ being a free parameter. Then, by using the multipliers identified in \cite{HS2024} and a delicate splitting of the viscous term, we showed that by choosing {\color{black} $\beta=3, \,6,\,9$} respectively for the second-, third- and fourth-order schemes, these schemes  are unconditionally stable in the absence of nonlinear terms. Then, we proved by induction  optimal global-in-time convergence rates in both 2D and 3D for the nonlinear Navier-Stokes equations. There results are the first stability and convergence results for any fully decoupled, higher-than second-order schemes for the Navier-Stokes equations.

 We provided numerical results to show that the third- and fourth-order schemes based on the usual BDF (i.e. $\beta=1$) are not unconditionally stable while the new third- and fourth-order schemes with  $\beta=\beta_k$ specified in \eqref{betak}  are unconditionally stable and lead to expected convergence rates.

Below are some problems related to this paper that deserve further investigation:
\begin{itemize}
\item We only carried out stability and error analysis for the second- to fourth-order consistent splitting schemes in this paper. It is still an open question whether these results can be extended to the fifth- and six-order consistent splitting schemes with suitable $\beta$.
\item We only considered semi-discrete (in time) schemes in this paper. It is worthwhile to construct suitable space discretizations for these consistent splitting schemes and carry out corresponding stability and error analysis. {\color{black} Note that if one uses a spectral method or $C^1$ finite-element method, it is expected that the results established in this paper can be directly extended to the fully discrete cases. However, the case with a $C^0$ finite-element method would be much more delicate as we can not directly test the scheme with $\Delta v_h$. We are currently working on a DG finite-element method to overcome this difficulty.}
\item {\color{black} A key element  for the stability analysis is Lemma 2 which requires $\Omega\in C^3$.  Our numerical results indicate that the proved stability and convergence rate  are still valid  in a square domain. However, it is not clear and beyond the scope of this paper whether the proof can be extended to polygonal domains.}
\item Since the Navier-Stokes equations are essential components of many coupled complex nonlinear systems, such as magneto-hydrodynamic equations, Navier-Stokes-Cahn-Hilliard equations, etc, it would be interesting to extend the results in this paper for Navier-Stokes equations to coupled complex nonlinear systems involving Navier-Stokes equations.
\end{itemize}

\begin{appendix}
\section{Proof of \eqref{Fk_ineq}}\label{app1}
Here, we provide the explicit telescoping forms for $\big(F_k^{\beta_k}(\bm u^{n+1}), C_k^{\beta_k}(\bm u^{n+1}) \big)$ and hence prove \eqref{Fk_ineq}:
\begin{equation}\label{F2_app}
\begin{split}
k=2,\,\beta_2=3: \quad & \big(F_2^3(\bm u^{n+1}),C_2^3(\bm u^{n+1}) \big) =\big(\frac{1}{100}\bm u^{n+1}, 4\bm u^{n+1}-3\bm u^{n} \big)\\
& =\frac{1}{100}\|\bm u^{n+1}\|^2+\frac{3}{200}\|\bm u^{n+1}\|^2-\frac{3}{200}\|\bm u^{n}\|^2+\frac{3}{200}\|\bm u^{n+1}-\bm u^n\|^2.
\end{split}
\end{equation}
\begin{equation}\label{F3_app}
\begin{split}
k=3,\,& \beta_2=6: \quad  \big(F_3^6(\bm u^{n+1}),C_3^6(\bm u^{n+1}) \big) =\big(\frac{27}{100}\bm u^{n+1}-\frac{21}{100}\bm u^{n}, 28\bm u^{n+1}-48\bm u^n+21\bm u^{n-1} \big)\\
 =&\big(\frac{27}{100}\bm u^{n+1}-\frac{21}{100}\bm u^{n}, \bm u^{n+1} \big)+\big(\frac{27}{100}\bm u^{n+1}-\frac{21}{100}\bm u^{n}, 27\bm u^{n+1}-48\bm u^n+21\bm u^{n-1} \big)\\
 =&\frac{3}{50}\|\bm u^{n+1}\|^2+\frac{21}{200}\|\bm u^{n+1}\|^2-\frac{21}{200}\|\bm u^{n}\|^2+\frac{21}{200}\|\bm u^{n+1}-\bm u^n\|^2\\
& +\frac{1}{200}\|27\bm u^{n+1}-21\bm u^n\|^2-\frac{1}{200}\|27\bm u^{n}-21\bm u^{n-1}\|^2+\frac{1}{200}\|27\bm u^{n+1}-48\bm u^n+21\bm u^{n-1}\|^2.
\end{split}
\end{equation}
\begin{equation}\label{F4_app}
\begin{split}
& k=4,\,\beta_2=9: \\
& \big(F_4^9(\bm u^{n+1}),C_4^9(\bm u^{n+1}) \big) \\
 =&\frac{2}{10^5}\big(215\bm u^{n+1}-375\bm u^{n}+165\bm u^{n-1}, 220\bm u^{n+1}-594\bm u^n+540\bm u^{n-1}-165\bm u^{n-2} \big)\\
 =&\frac{2}{10^5}\big(215\bm u^{n+1}-375\bm u^{n}+165\bm u^{n-1}, 5\bm u^{n+1}-4\bm u^n \big)\\
& +\frac{2}{10^5}\big(215\bm u^{n+1}-375\bm u^{n}+165\bm u^{n-1}, 215\bm u^{n+1}-590\bm u^n+540\bm u^{n-1}-165\bm u^{n-2} \big)\\
 =&\frac{2}{10^5}\big(210\bm u^{n+1}-375\bm u^{n}+165\bm u^{n-1}, 5\bm u^{n+1}-4\bm u^n \big)+\frac{2}{10^5}\big(5\bm u^{n+1}, 5\bm u^{n+1}-4\bm u^n \big)\\
& +\frac{1}{10^5}\big(\|215\bm u^{n+1}-375\bm u^{n}+165\bm u^{n-1} \|^2-\|215\bm u^{n}-375\bm u^{n-1}+165\bm u^{n-2} \|^2\big)\\
&+\frac{1}{10^5}\| 215\bm u^{n+1}-590\bm u^n+540\bm u^{n-1}-165\bm u^{n-2}\|^2,\\
=&\frac{2}{10^5}\big(210\bm u^{n+1}-375\bm u^{n}+165\bm u^{n-1}, 5\bm u^{n+1}-4\bm u^n \big)+\frac{1}{10^4}\|\bm u^{n+1}\|^2\\
&+\frac{2}{10^4}\big( \|\bm u^{n+1}\|^2-\|\bm u^n\|^2+\|\bm u^{n+1}-\bm u^n\|^2\big)\\
&+\frac{1}{10^5}\big(\|215\bm u^{n+1}-375\bm u^{n}+165\bm u^{n-1} \|^2-\|215\bm u^{n}-375\bm u^{n-1}+165\bm u^{n-2} \|^2\big)\\
&+\frac{1}{10^5}\| 215\bm u^{n+1}-590\bm u^n+540\bm u^{n-1}-165\bm u^{n-2}\|^2,
\end{split}
\end{equation}
and finally, for the term $\big(210\bm u^{n+1}-375\bm u^{n}+165\bm u^{n-1}, 5\bm u^{n+1}-4\bm u^n \big)$, we have
\begin{equation}
\begin{split}
&\big(210\bm u^{n+1}-375\bm u^{n}+165\bm u^{n-1}, 5\bm u^{n+1}-4\bm u^n \big)\\
&=a\|\bm u^{n+1}\|^2-a\|\bm u^{n}\|^2+\|b\bm u^{n+1}+c\bm u^n\|^2-\|b\bm u^{n}+c\bm u^{n-1}\|^2
+\|d\bm u^{n+1}+e\bm u^n+f\bm u^{n-1}\|^2.
\end{split}
\end{equation}
with
\begin{small}
\begin{equation}
 e=-\sqrt{\frac{3375}{2}},\,f=\frac{-\sqrt{37.5}+\sqrt{1687.5}}{2},\,c=f,\,d=\sqrt{37.5}+f,\,b=\frac{660+2ef}{2c},\,a=1050-b^2-d^2\approx 0.2188.
\end{equation}
\end{small}
\section{Proof of Lemma \ref{DC_lemma}}\label{app2}
\end{appendix}
\begin{proof}
The proof follows the basic process as in \cite{akrivis2021energy}. We consider the case $k=2,3,4$ separately.\\
\textbf{Case I: k=2.} With $c_{2,q}$ obtained from \eqref{solve_ckq} and $d_{2,q}$ defined in \eqref{dkq} and $\beta_2=3$, $\eta_k=0.71$, the explicit form of $\tilde{C}_2^3(\zeta)$ and $\tilde{D}_2^3(\zeta)$ are given as
\begin{equation}
\tilde{C}_2^3(\zeta)=4\zeta-3,\quad \tilde{D}_2^3(\zeta)=\frac{3}{20}\zeta+\frac{13}{100},
\end{equation}
which imply $\tilde{C}_2^3(\frac{3}{4})=0$ and $ \tilde{D}_2^3(\frac{-13}{15})=0$. Hence $\tilde{C}_2^3(\zeta)$, $\tilde{D}_2^3(\zeta)$ have no common divisor and $\frac{\tilde{D}_2^3(\zeta)}{\tilde{C}_2^3(\zeta)}$ is holomorphic outside the unit disk. Moreover, we have
\begin{equation}
\lim_{|\zeta| \rightarrow \infty}\frac{\tilde{D}_2^3(\zeta)}{ \tilde{C}_2^3(\zeta)}=\frac{3}{80}>0.
\end{equation}
Therefore, it follows from the maximum principle for harmonic functions, $\text{Re}\frac{\tilde{D}_2^3(\zeta)}{\tilde{C}_2^3(\zeta)}>0,\quad \forall |\zeta|>1$ is equivalent to
\begin{equation}\label{A10}
\text{Re} \frac{\tilde{D}_2^3(\zeta)}{\tilde{C}_2^3(\zeta)}\ge 0, \quad \forall \zeta \in \mathbb{S}^1,
\end{equation}
with $\mathbb{S}^1$ is the unit circle in the complex plane and \eqref{A10} is equivalent to
\begin{equation}\label{A11}
Re[\tilde{D}_2^3(e^{i\theta})\tilde{C}_2^3(e^{-i\theta})] \ge 0, \quad \theta \in [0, 2\pi).
\end{equation}
Denote $y:=\cos(\theta)$, then \eqref{A11} is equivalent to
\begin{equation}
Re[\tilde{D}_2^3(e^{i\theta})\tilde{C}_2^3(e^{-i\theta})]=\frac{7}{100}y+\frac{21}{100} \ge 0, \quad  \forall y \in [-1, 1],
\end{equation}
which is obvious true and hence we proved Lemma \ref{DC_lemma} with $k=2$.\\
\textbf{Case II: k=3.} With $k=3$ and $\beta_3=6$, the explicit form of $\tilde{C}_3^6(\zeta)$ and $\tilde{D}_3^6(\zeta)$ are given as
\begin{equation}
\tilde{C}_3^6(\zeta)=28\zeta^2-48\zeta+21,\quad \tilde{D}_3^6(\zeta)=\frac{17}{20}\zeta^2-\frac{71}{100}\zeta+\frac{9}{100},
\end{equation}
and the zeros of $\tilde{C}_3^6(\zeta)$ are $\frac{12 \pm \sqrt{3}i}{14}$, the zeros of $\tilde{D}_3^6(\zeta)$ are $\frac{71 \pm \sqrt{1981}}{170}$, which imply $\tilde{C}_3^6(\zeta)$, $\tilde{D}_3^6(\zeta)$ have no common divisor and $|\frac{12 \pm \sqrt{3}i}{14}|<1$ implies $\frac{\tilde{D}_3^6(\zeta)}{\tilde{C}_3^6(\zeta)}$ is holomorphic outside the unit disk. Following the same process as the second order case, one can easily show $\text{Re}\frac{\tilde{D}_3^6(\zeta)}{\tilde{C}_3^6(\zeta)}>0,\quad \forall |\zeta|>1$ is equivalent to
\begin{equation}
f_3(y):=\frac{2037}{50}y^2-\frac{7991}{100}y+\frac{197}{5} \ge 0, \quad  \forall y \in [-1, 1],
\end{equation}
which is true since
\begin{equation}
\mathop{\rm {min}}\limits_{y \in [-1,1]} f_3(y)=f_3(\frac{7991}{8148})\approx 0.214874>0.
\end{equation}
\textbf{Case III: k=4.} With $k=4$ and $\beta_4=9$, the explicit form of $\tilde{C}_4^9(\zeta)$ and $\tilde{D}_4^9(\zeta)$ are given as
\begin{subequations}
\begin{align}
&\tilde{C}_4^9(\zeta)=220\zeta^3-594\zeta^2+540\zeta-165,\\
&\tilde{D}_4^9(\zeta)=\frac{1}{10^4}(87957\zeta^3-182525\zeta^2+125967\zeta-28500),
\end{align}
\end{subequations}
and the zeros of $\tilde{C}_4^9(\zeta)$ (with six decimal places) are
\begin{equation}
\zeta_{C1}=0.858473,\,\zeta_{C2}=0.920763+0.160745i,\,\zeta_{C3}=0.920763-0.160745i,
\end{equation}
and the zeros of $\tilde{D}_4^9(\zeta)$ (with six decimal places) are
\begin{equation}
\zeta_{D1}=0.517951,\,\zeta_{D2}=0.778605+0.139132i,\,\zeta_{D3}=0.778605-0.139132i,
\end{equation}
 which imply $\tilde{C}_4^9(\zeta)$, $\tilde{D}_4^9(\zeta)$ have no common divisor and $|\zeta_{Ci} |<1,\,i=1,2,3$ implies $\frac{\tilde{D}_4^9(\zeta)}{\tilde{C}_4^9(\zeta)}$ is holomorphic outside the unit disk. Following the same process as the second order case, one can easily show $\text{Re}\frac{\tilde{D}_4^9(\zeta)}{\tilde{C}_4^9(\zeta)}>0,\quad \forall |\zeta|>1$ is equivalent to
\begin{equation}\label{f4}
f_4(y):=\alpha_3y^3+\alpha_2y^2+\alpha_1y+\alpha_0 \ge 0, \quad  \forall y \in [-1, 1],
\end{equation}
with
\begin{equation}
\alpha_3=-\frac{429\times9689}{500},\,\alpha_2=\frac{9\times2716781}{1000},\,\alpha_1=-\frac{241\times62141}{625},\,\alpha_0=\frac{53\times3^{10}}{400}.
\end{equation}
\eqref{f4} is true since
\begin{equation}
\mathop{\rm {min}}\limits_{y \in [-1,1]} f_4(y)=f_4(y^*) \approx 3.000376\times 10^{-4} >0.
\end{equation}
with $y^*=\frac{-\alpha_2+\sqrt{\alpha_2^2-3\alpha_1\alpha_3}}{3\alpha_3} \approx  0.959828$.
The proof for all the cases is completed.
\end{proof}

\bibliographystyle{plain}
\bibliography{bib_NS_error,consistent_splitting}

\begin{thebibliography}{10}

\bibitem{akrivis2021energy}
G.~Akrivis, M.~Chen, F.~Yu, and Z.~Zhou.
\newblock The energy technique for the six-step {BDF} method.
\newblock {\em SIAM J. Numer. Anal.}, 59(5):2449--2472, 2021.

\bibitem{brezzi2012mixed}
F.~Brezzi and M.~Fortin.
\newblock {\em Mixed and hybrid finite element methods}, volume~15.
\newblock Springer Science \& Business Media, 2012.

\bibitem{dahlquist1978g}
G.~Dahlquist.
\newblock G-stability is equivalent to {A}-stability.
\newblock {\em BIT Numer. Math.}, 18:384--401, 1978.

\bibitem{weinan1995projection}
W.~E and J.~G. Liu.
\newblock Projection method {I}: convergence and numerical boundary layers.
\newblock {\em SIAM J. Numer. Anal.}, pages 1017--1057, 1995.

\bibitem{E2003Gauge}
W.~E and J.~G. Liu.
\newblock Gauge method for viscous incompressible flows.
\newblock {\em Commun. Math. Sci.}, 1:317--332, 2003.

\bibitem{elman2014finite}
H.~C. Elman, D.~J. Silvester, and A.~J. Wathen.
\newblock {\em Finite elements and fast iterative solvers: with applications in
  incompressible fluid dynamics}.
\newblock Oxford University Press, USA, 2014.

\bibitem{girault1979finite}
V.~Girault and P.~A. Raviart.
\newblock Finite element approximation of the {N}avier-{S}tokes equations.
\newblock {\em Lecture Notes in Mathematics, Berlin Springer Verlag}, 749,
  1979.

\bibitem{guermond2012convergence}
J.~L. Guermond, P.~Minev, and A.~Salgado.
\newblock Convergence analysis of a class of massively parallel direction
  splitting algorithms for the {N}avier-{S}tokes equations in simple domains.
\newblock {\em Math. Comput.}, 81(280):1951--1977, 2012.

\bibitem{guermond2006overview}
J.~L. Guermond, P.~Minev, and J.~Shen.
\newblock An overview of projection methods for incompressible flows.
\newblock {\em Comput. Methods Appl. Mech. Eng.}, 195(44-47):6011--6045, 2006.

\bibitem{guermond2011error}
J.~L. Guermond and A.~J. Salgado.
\newblock Error analysis of a fractional time-stepping technique for
  incompressible flows with variable density.
\newblock {\em SIAM J. Numer. Anal.}, 49(3):917--944, 2011.

\bibitem{guermond2003velocity}
J.~L. Guermond and J.~Shen.
\newblock Velocity-correction projection methods for incompressible flows.
\newblock {\em SIAM J. Numer. Anal.}, 41(1):112--134, 2003.

\bibitem{guermond2004error}
J.~L. Guermond and J.~Shen.
\newblock On the error estimates for the rotational pressure-correction
  projection methods.
\newblock {\em Math. Comput.}, 73(248):1719--1737, 2004.

\bibitem{Shen03JCP}
J.L. Guermond and J.~Shen.
\newblock A new class of truly consistent splitting schemes for incompressible
  flows.
\newblock {\em J. Comput. Phys.}, 192(1):262--276, 2003.

\bibitem{he1998NM}
Y.~He and K.~Li.
\newblock Convergence and stability of finite element nonlinear {G}alerkin
  method for the {N}avier-{S}tokes equations.
\newblock {\em Numer. Math.}, 79(1):77--106, 1998.

\bibitem{HS22}
F.~Huang and J.~Shen.
\newblock A new class of implicit--explicit {BDF}k sav schemes for general
  dissipative systems and their error analysis.
\newblock {\em Comput. Methods Appl. Mech. Eng.}, 392:114718, 2022.

\bibitem{HS2023}
F.~Huang and J.~Shen.
\newblock Stability and error analysis of a second order consistent splitting
  scheme for the {N}avier-{S}tokes equation.
\newblock {\em SIAM J. Numer. Anal.}, 61(5):2408--2433, 2023.

\bibitem{HS2024}
F.~Huang and J.~Shen.
\newblock On a new class of {BDF} and {IMEX} schemes for parabolic type
  equations.
\newblock {\em SIAM J. Numer. Anal.}, 62(4):1609--1637, 2024.

\bibitem{johnston2004accurate}
H.~Johnston and J.~G. Liu.
\newblock Accurate, stable and efficient {N}avier-{S}tokes solvers based on
  explicit treatment of the pressure term.
\newblock {\em J. Comput. Phys.}, 199:221--259, 2004.

\bibitem{karniadakis1991high}
G.~E. Karniadakis, M.~Israeli, and S.~A. Orszag.
\newblock High-order splitting methods for the incompressible {N}avier-{S}tokes
  equations.
\newblock {\em J. Comput. Phys.}, 97(2):414--443, 1991.

\bibitem{liuCPAM}
J.-G. Liu, J.~Liu, and R.~L. Pego.
\newblock Stability and convergence of efficient {N}avier-{S}tokes solvers via
  a commutator estimate.
\newblock {\em Comm. Pure Appl. Math}, 60(10):1443--1487, 2007.

\bibitem{nevanlinna1981}
O.~Nevanlinna and F.~Odeh.
\newblock Multiplier techniques for linear multistep methods.
\newblock {\em Numerical Functional Analysis and Optimization}, 3(4):377--423,
  1981.

\bibitem{Nochetto2003Gauge}
R.~Nochetto and J.-H. Pyo.
\newblock Error estimates for semi-discrete gauge methods for the
  {N}avier-{S}tokes equations.
\newblock {\em Math. Comput.}, 74:521--542, 2005.

\bibitem{orszag1986boundary}
S.~A. Orszag, M.~Israeli, and M.~O. Deville.
\newblock Boundary conditions for incompressible flows.
\newblock {\em J. Sci. Comput.}, 1(1):75--111, 1986.

\bibitem{prohl1997projection}
A.~Prohl.
\newblock {\em Projection and quasi-compressibility methods for solving the
  incompressible {N}avier-{S}tokes equations}.
\newblock Springer, 1997.

\bibitem{shen1992error}
J.~Shen.
\newblock On error estimates of projection methods for {N}avier-{S}tokes
  equations: first-order schemes.
\newblock {\em SIAM J. Numer. Anal.}, 29(1):57--77, 1992.

\bibitem{Shen2012Modeling}
J.~Shen.
\newblock Modeling and numerical approximation of two-phase incompressible
  flows by a phase-field approach.
\newblock In {\em Multiscale modeling and analysis for materials simulation},
  pages 147--195. World Scientific, 2012.

\bibitem{STWbook}
J.~Shen, T.~Tang, and L.~L. Wang.
\newblock {\em Spectral methods: algorithms, analysis and applications},
  volume~41.
\newblock Springer Science \& Business Media, 2011.

\bibitem{shen2007error}
J.~Shen and X.~Yang.
\newblock Error estimates for finite element approximations of consistent
  splitting schemes for incompressible flows.
\newblock {\em Discrete Contin. Dyn. Syst. - B.}, 8(3):663--675, 2007.

\bibitem{Tema83}
R.~Temam.
\newblock {\em {N}avier-{S}tokes Equations and Nonlinear Functional Analysis}.
\newblock SIAM, Philadelphia, 1983.

\bibitem{timmermans1996approximate}
L.~J. Timmermans, P.~D. Minev, and F.~N. Van De~Vosse.
\newblock An approximate projection scheme for incompressible flow using
  spectral elements.
\newblock {\em Int. J. Numer. Methods Fluids}, 22(7):673--688, 1996.

\bibitem{wang2000convergence}
C.~Wang and J.~G. Liu.
\newblock Convergence of gauge method for incompressible flow.
\newblock {\em Math. Comput.}, 69(232):1385--1407, 2000.

\bibitem{WHSJCP22}
K.~Wu, F.~Huang, and J.~Shen.
\newblock A new class of higher--order decoupled schemes for the incompressible
  {N}avier--{S}tokes equations and applications to rotating dynamics.
\newblock {\em J. Comput. Phys.}, 458:111097, 2022.

\end{thebibliography}
\end{document}